\documentclass[11pt]{amsart}

\usepackage{amsmath,amsthm,amssymb,amsxtra,amsfonts,amsbsy,mathrsfs}
\usepackage{verbatim}
\usepackage{hyperref}
\usepackage[all]{xy}
\usepackage{url}
\newtheorem{theorem}{Theorem}[section]

\newtheorem{lemma}[theorem]{Lemma}
\newtheorem{proposition}[theorem]{Proposition}
\newtheorem{definition}[theorem]{Definition}

\newtheorem{remark}[theorem]{Remark}
\newtheorem{example}[theorem]{Example}

\newtheorem*{maintheorem}{Theorem \ref{thm:main}}

\renewcommand{\setminus}{\smallsetminus}
\newcommand{\Q}{\mathbb{Q}}
\newcommand{\Z}{\mathbb{Z}}

\newcommand{\N}{\mathbb{N}}

\newcommand{\mf}[1]{\mathfrak{#1}}
\newcommand{\OO}{\mathcal{O}}
\newcommand{\kVG}{k[V]^G}
\newcommand{\CV}{C(P)^{\vee}}
\newcommand{\MM}{\mathcal{M}}
\newcommand{\NN}{\mathcal{N}}

\newcommand{\bbar}[1]{\setbox0=\hbox{$#1$}\dimen0=.2\ht0 \kern\dimen0 \overline{\kern-\dimen0 #1}}

\newcommand{\ra}{\to}

\DeclareMathOperator{\Spec}{Spec}

\DeclareMathOperator{\coker}{coker}
\DeclareMathOperator{\Hom}{Hom}
\DeclareMathOperator{\Ext}{Ext}

\DeclareMathOperator{\rank}{rank}
\DeclareMathOperator{\Cone}{Cone}

\SelectTips{xy}{12}

\makeatletter
\@namedef{subjclassname@2010}{ \textup{2010} Mathematics Subject 
Classification} \makeatother

\begin{document}


\begin{abstract}
We develop a theory of toric Artin stacks extending the theories of toric Deligne-Mumford stacks developed by Borisov-Chen-Smith, Fantechi-Mann-Nironi, and Iwanari.  We also generalize the Chevalley-Shephard-Todd theorem to the case of diagonalizable group schemes. These are both applications of our main theorem which shows that a toroidal embedding $X$ is canonically the good moduli space (in the sense of Alper) of a smooth log smooth Artin stack whose stacky structure is supported on the singular locus of $X$.
\end{abstract}

\title{Canonical Artin Stacks over Log Smooth Schemes}
\author{Matthew Satriano}
\thanks{Department of Mathematics, University of Michigan, 2074 East Hall, Ann Arbor, MI 48109- 1043, USA.\\
E-mail: \url{satriano@umich.edu}}
\date{}
\maketitle

\vspace{-1cm}



\section{Introduction}
\label{sec:intro}
Resolution of singularities is a fundamental tool in the study of singular schemes, yet resolutions are oftentimes difficult to control.  In certain cases, however, one can introduce a canonical smooth stack which approximates a singular scheme $X$ and serves as a replacement for a resolution of singularities.  The first known instance of this is when $X$ is a scheme over a field $k$ with quotient singularities prime to the characteristic of $k$: for such $X$, there is a smooth Deligne-Mumford stack $\mf{X}$ and a coarse space map $f:\mf{X}\to X$ such that the base change of $f$ to the smooth locus of $X$ is an isomorphism (see for example \cite[2.9]{int}).

When the singularities of $X$ are worse than quotient singularities, it is not possible to find 
a Deligne-Mumford stack $\mf{X}$ as above.  One can still hope, however, to ``resolve the singularities'' of $X$ by a smooth Artin stack.  Typically Artin stacks do not have coarse spaces, and the appropriate notion that replaces coarse space is that of good moduli space (in the sense of \cite{alper}).

Inspired by the work of Iwanari \cite{mfr}, we restrict attention to a class of schemes which carry more structure in hopes of being able to both construct a stacky resolution and say more about it than we could for an arbitrary scheme.  This richer class of schemes we look at is that of fs log smooth log schemes $X$ over $k$, where $\Spec k$ is given the trivial log structure (or equivalently, the class of toroidal embeddings which are not necessarily strict).  Our main result is then:

\begin{maintheorem}
Let $k$ be field and $X$ be an fs log scheme which is log smooth over $S=\Spec k$, where $S$ is given the trivial log structure.  Then there exists a smooth log smooth Artin stack $\mf{X}$ over $S$ and a good moduli space morphism $f:\mf{X}\rightarrow X$ over $S$.  Moreover, the base change of $f$ to the smooth locus of $X$ is an isomorphism.
\end{maintheorem}
This is a generalization of \cite[Thm 3.3]{mfr} where the result is proved for $X$ all of whose charts are given by simplicial toric varieties.  The stack $\mf{X}$ has a moduli interpretation in terms of log geometry and agrees with the stack Iwanari constructs when $X$ is as in \cite[Thm 3.3]{mfr}.  
In addition to having a moduli interpretation, the stacky resolution $f$ has the property that it is an isomorphism away from codimension 2 on the \emph{source}. This property was particularly useful, for example, in \cite{dR}: in \cite[Thm 1.10]{cst}, we constructed analogous stacky resolutions for schemes with linearly reductive singularities (\cite[Def 5.1]{cst}), and then used such resolutions in \cite[Thm 4.8]{dR} to show the degeneracy of a variant of the Hodge-de Rham spectral sequence in characteristic $p$.

We give two applications of Theorem \ref{thm:main}: a generalization of the Chevalley-Shephard-Todd theorem (\cite[\S 5 Thm 4]{bourbaki}) to the case of diagonalizable group schemes and a generalization of the work of \cite{bcs,fant,iwtoric} on toric Deligne-Mumford stacks to the case of toric Artin stacks.

\subsection*{Chevalley-Shephard-Todd Theorem for diagonalizable group schemes}
We recall that if $k$ is a field and 
$G$ is a finite (abstract) group which acts faithfully on a $k$-vector space $V$, then $g\in G$ is called a \emph{pseudo-reflection} if $V^g$ is a hyperplane.  The Chevalley-Shephard-Todd Theorem states that if the order of $G$ is prime to the characteristic of $k$, then the invariants $k[V]^G$ is a polynomial ring if and only if $G$ is generated by pseudo-reflections. 

There is a strong connection between the Chevalley-Shephard-Todd theorem and the 
the existence of stacky resolutions for schemes with quotient singularities prime to the characteristic. 
In fact, we hope to convince the reader that (at least in the setting of linearly reductive group schemes) proving stacky resolution theorems as above is roughly \emph{equivalent} to proving Chevalley-Shephard-Todd type theorems.  This philosophy is demonstrated, for example, in \cite{cst} where we generalized the Chevalley-Shephard-Todd theorem to the case of finite linearly reductive group schemes and from this 
we deduced a stacky resolution theorem \cite[Thm 1.10]{cst} for schemes with linearly reductive singularities. Conversely, in this paper, we \emph{first} prove a stacky resolution theorem for toroidal embeddings (Theorem \ref{thm:main}) and then by reinterpreting its proof, arrive at a Chevalley-Shephard-Todd theorem for diagonalizable group schemes $G$ (see Theorem \ref{thm:cst}, which is a bit too technical to state in the introduction).

We remark that in the process of proving 
Theorem \ref{thm:cst}, we give necessary and sufficient conditions for when 
$\Spec(\kVG)$ is a simplicial toric variety (see Theorem \ref{thm:msopsimp}, also obtained by Wehalu in \cite[Thm 4.1]{wehlau}). We note that Theorem \ref{thm:cst} recovers 
\cite[Thm 5.6]{wehlau}, which gives the result when $G$ is a torus.

\subsection*{Toric Artin stacks}
If we require that $X$ is a toric variety rather than just a toroidal embedding, then by considering variants of the moduli problem in Theorem \ref{thm:main}, we produce (Theorem \ref{thm:toricmain}) many smooth log smooth Artin stacks, all of which have $X$ as a good moduli space. Each of these Artin stacks has a dense open torus whose action on itself extends to an action on the stack.  Thus, these stacks should all be thought of as toric Artin stacks.  

In \cite{bcs}, Borisov, Chen, and Smith introduce the notion of a stacky fan $\mathbf{\Sigma}$
and associate to it a smooth Deligne-Mumford stack $\mf{X}(\mathbf{\Sigma})$.  They refer to stacks obtained in this way as toric Deligne-Mumford stacks. Fantechi, Mann, and Nironi \cite{fant} (and independently Iwanari \cite{iwtoric}) took a different approach to toric stacks.  They define a toric Deligne-Mumford stack to be a smooth Deligne-Mumford stack $\mf{X}$ with a dense torus $T$ whose action on itself extends to an action on $\mf{X}$.  They then show that this definition is equivalent to the one introduced in \cite{bcs}; that is, every such stack is of the form $\mf{X}(\mathbf{\Sigma})$.

Analogously, one might hope that the moduli-theoretic Artin stacks we produce in Theorem \ref{thm:toricmain} also come from some sort of stacky fan.  This is indeed the case.  In Section \ref{subsec:stackyfan}, we introduce a notion of stacky fan which generalizes that of \cite{bcs}, and associate to each generalized stacky fan a smooth log smooth Artin stack.  We then show in Theorem \ref{thm:twostacks} that our moduli-theoretic stacks of Theorem \ref{thm:toricmain} all come from generalized stacky fans.  We hope that Section \ref{sec:toric} will form the basis for a theory of toric Artin stacks, extending the theories of toric Deligne-Mumford stacks introduced in \cite{bcs}, \cite{fant}, and \cite{iwtoric}.

\subsection*{Notation and prerequisites.}
We assume the reader is familiar with log geometry as in \cite{kkato}.  For basic properties concerning monoids, we refer the reader to \cite[Chpt I]{ogus}. Given a monoid $P$ and a ring $R$, we denote the standard log structure on $\Spec R[P]$ by $\MM_P$. If $X$ is a log scheme, we denote by $X^{triv}$ the locus where the log structure is trivial.


Given a geometric point $\bar{x}$ of a scheme $X$, we let $\OO_{X,\bar{x}}$ denote the strict henselization of $\OO_{X,x}$, where $x$ is the image of $\bar{x}$ in $X$.

Given an abelian group $A$ and a scheme $X$, we denote by $D_X(A):=\Spec_X \OO_X[A]$ the diagonalizable group scheme over $X$ associated to $A$.  We often drop the subscript $X$ when it is understood from context.

\section{Minimal Free Resolutions}
\label{sec:mfr}
In this section, we define the objects parameterized by the stack $\mf{X}$ of Theorem \ref{thm:main}.  We begin by recalling some definitions. A monoid $P$ is called \emph{toric} if it is fine, saturated, and $P^{gp}$ is torsion free.  We denote by $C(P)$ the rational cone generated by $P$ in $P^{gp}\otimes\Q$.  A toric monoid $P$ is called \emph{simplicially toric} if the number of extremal rays of the dual cone $\CV$ is equal to the rank of $P^{gp}$.

A morphism $f:P\ra Q$ of integral monoids is called \emph{close} if for all $q\in Q$, there exists a positive integer $n$ such that $nq$ is in the image of $f$. 
We recall (\cite[Def 2.6]{mfr}) that an injective morphism $i:P\to F$ from a simplicially toric sharp 
monoid to a free monoid is called a \emph{minimal free resolution} if $i$ is close and if for all injective close 
morphisms $i':P\to F'$ to a free monoid $F'$ of the same rank as $F$, there 
is a unique morphism $j:F\to F'$ such that $i'=ji$.  The stack Iwanari constructs in \cite[Thm 3.3]{mfr} 
has a moduli interpretation in terms of minimal free resolutions, so our first step is to generalize this notion.  
The key is to replace his use of closeness in the above definition with that of exactness.
\begin{definition}[{\cite[Def I.2.1.8]{ogus}}]
\emph{A morphism $f:P\to Q$ of integral monoids is \emph{exact} if the diagram}
\[
\xymatrix{
P\ar[r]^{f}\ar[d] & Q\ar[d]\\
P^{gp}\ar[r]^{f^{gp}} & Q^{gp}
}
\]
\emph{is set-theoretically cartesian.}
\end{definition}
Note that if $f:P\to Q$ is sharp and exact, then it is automatically injective.\\
\\
Let $P$ be a  toric sharp monoid.  Let $\rho_1,\dots,\rho_d$ be the extremal rays of $C(P)^{\vee}$ and denote by $v_i$ the first lattice point on $\rho_i$. Let $F(P)$ be the free monoid on the $v_i$.  We obtain a morphism $i:P\to F(P)$ defined by 
$p\mapsto (v_i(p))$.
\begin{proposition}
\label{prop:can}
The morphism $i:P\to F(P)$ is exact.  Moreover, for any exact morphism $i':P\to F$ to a free monoid $F$ of rank $d$, there is a unique morphism $j:F(P)\to F$ such that $i'=ji$.
\end{proposition}
\begin{proof}
We begin by showing that $i$ is exact.  This is equivalent to showing that if $p\in P^{gp}$ and $v_i(p)\geq0$ for all $i$, then $p\in P$.  This, in turn, is equivalent to the statement that $P=(C(P)^{\vee})^{\vee}\cap P^{gp}$.  This latter statement follows, for example, from \cite[p.9 (1)]{fulton}.

We now show that $i$ is universal.  Fix an isomorphism $F\cong\N^d$. Then $i'=(\varphi_i)$. Exactness of $i'$ shows that an element $p$ of $P^{gp}$ is in $P$ if and only if $\varphi_i^{gp}(p)\geq 0$ for all $i$.  Since $p$ is in $P$ if and only if $f(p)\geq0$ for all $f\in C(P)^{\vee}$, we see 
\[
\Cone(\varphi_1,\dots,\varphi_d)=\CV=\Cone(v_1,\dots,v_d),
\]
where $\Cone(w_1,\dots,w_d)$ denotes the cone generated by the lattice points $w_i\in\Hom(P^{gp},\Z)$.  Since $\CV$ has exactly $d$ extremal rays, it follows that each $\varphi_i$ lies on a distinct extremal ray.  Composing $i'$ by a uniquely determined permutation of coordinates, we can assume that $\varphi_i$ lies on the ray $\rho_i$.  Since $\varphi_i$ is a lattice point of $\Hom(P^{gp},\Q)$ (as $\varphi_i$ takes integer values on $P$), and since $v_i$ is defined to be the first lattice point on the ray $\rho_i$, we have $\varphi_i=n_iv_i$ for uniquely determined $n_i\in\N$.  Hence, $\sum a_iv_i\mapsto (n_1a_1,\dots,n_da_d)$ is our desired $j$.
\end{proof}
In light of this proposition, we make the following definition.
\begin{definition}
\label{def:mfr}
\emph{Let $P$ be a  toric sharp monoid and let $d$ be the number of extremal rays of $\CV$.  A morphism $i:P\to F$ to a free monoid of rank $d$ is a 
\emph{minimal\ free\ resolution} if it is exact and if for every exact morphism $i':P\to F'$ to a free 
monoid $F'$ of rank $d$, there is a unique morphism $j:F\to F'$ such that $i'=ji$.}
\end{definition}
If all of the charts of $X$ are given by simplicial toric varieties (as in the case Iwanari considers), then one need only consider minimal free resolutions when constructing $\mf{X}$. In the non-simplicial case, however, $\mf{X}$ parametrizes morphisms obtained by ``slicing'' a minimal free resolution by a suitable face:
\begin{definition}
\emph{Let $P$ be a toric sharp monoid. A morphism of monoids $i':P\to F'$ is a \emph{sliced resolution} if there is a minimal free resolution $i:P\to F$, a face $H$ of $F$ such that $i(P)\cap H=0$, and an isomorphism of monoids $F'\to F/H$ over $P$.}
\end{definition}
If $P$ is a  simplicially toric sharp monoid, then the number of extremal rays of $\CV$ is equal to the rank of $P^{gp}$.  
If $i:P\to F$ is a morphism to a free monoid of rank equal to that of $P^{gp}$, then $i$ is exact if and only if 
it is injective and close.  Using this observation, it is an easy exercise to show that our definitions agree with those of Iwanari in the simplicial case:
\begin{lemma}
\label{l:all-definitions-agree}
Let $P$ be a simplicially toric sharp monoid and let $i:P\to F$ be a morphism to a free monoid of rank equal that of $P^{gp}$. Then the following are equivalent:
\begin{enumerate}
\item $i$ is a minimal free resolution in the sense of \cite[Def 2.6]{mfr}.
\item $i$ is a minimal free resolution in the sense of Definition \ref{def:mfr}.
\item $i$ is a sliced resolution.
\end{enumerate}
\end{lemma}
We now globalize the above definitions to morphisms of log schemes.
\begin{definition}
\emph{A morphism $f:(Y,\MM_Y)\to (X,\MM_X)$ of fs log schemes is a \emph{minimal\ free\ resolution}, resp.~a \emph{sliced resolution} if for all geometric points $\bar{y}$ of $Y$, the induced morphism}
\[
\bbar{\MM}_{X,f(\bar{y})}\longrightarrow \bbar{\MM}_{Y,\bar{y}}
\]
\emph{is a minimal free resolution, resp.~a sliced resolution.}
\end{definition}
\begin{proposition}
\label{prop:face}
Let $P$ be a toric sharp monoid with minimal free resolution $i:P\to F$.  If $P_0$ is a face of 
$P$ and $F_0$ is the face of $F$ generated by $P_0$, then the induced morphism $P/P_0\to F/F_0$ is a minimal 
free resolution.
\end{proposition}
\begin{proof}
Since any two minimal free resolutions of $P$ are uniquely isomorphic over $P$, we can assume $F=F(P)$ and $i(p)=(v_i(p))$, where $v_1,\dots,v_d$ are the first lattice points on the extremal rays of $\CV$. Let $\pi:P\to P/P_0$ be the natural projection.

One easily checks that $P/P_0$ is a toric sharp monoid, and so it makes sense to speak of its minimal free resolution.  If $v_i(P_0)=0$, then we obtain a well-defined morphism $w_i:P/P_0\to\N$ given by $w_i(\bar{p})=v_i(p)$.  Note that $w_i$ is an extremal ray of $C(P/P_0)^{\vee}$.  Indeed, if $\psi,\psi'\in C(P/P_0)^{\vee}$ such that $\psi+\psi'=w_i$, then $\psi\pi+\psi'\pi=v_i$.  Since $v_i$ is an extremal ray of $\CV$, it follows that $\psi\pi=av_i$ for some $a\in\Q_{\geq0}$, and so $\psi=aw_i$.

We claim that $C(P/P_0)^{\vee}$ is generated as a rational cone by the $w_i$ such that $v_i(P_0)=0$.  Let $\psi\in C(P/P_0)^{\vee}$.  Then $\psi\pi=\sum_j a_jv_j$ for $a_j\in \Q_{\geq0}$.  To prove our claim, we need only show $a_\ell=0$ if $v_\ell(P_0)\neq0$.  If $v_\ell(p_0)\neq0$ for some $p_0\in P_0$, then since
\[
0=\psi\pi(p_0)=\sum_j a_jv_j(p_0),
\]
and all $a_j\geq0$, we see $a_\ell v_\ell(p_0)=0$ so $a_\ell=0$, and our claim follows.

We next show that $F_0$ is the face generated by the $v_i$ with $v_i(P_0)\neq0$.  If $v_i(p_0)\neq0$ for some $p_0\in P_0$, then $v_i(p_0)>0$.  As a result,
\[
v_i+((v_i(p_0)-1)v_i+\sum_{j\neq i}v_j(p_0)v_j)=i(p_0)
\]
is in the image of $P_0$, so $v_i$ is in the face generated by $P_0$; that is, $v_i\in F_0$.  Conversely, if $v_i\in F_0$, then by the definition of $F_0$, there exists some $p_0\in P_0$ and $b_j\in \N$ such that 
\[
v_i+\sum_{j}b_jv_j=i(p_0)=\sum_{j}v_j(p_0)v_j.
\]
Hence, $v_i(p_0)\neq0$.

To complete the proof of the proposition, note that the induced map $P/P_0\ra F/F_0$ is given by
\[
\bar{p}\longmapsto\sum_{v_i(P_0)=0}v_i(p)\bar{v}_i=\sum_{v_i(P_0)=0} w_i(\bar{p})\bar{v}_i.
\]
Since the $w_i$ such that $v_i(P_0)=0$ are the distinct extremal rays of the rational cone $C(P/P_0)^{\vee}$, we see that $P/P_0\ra F/F_0$ is precisely the minimal free resolution $P/P_0\ra F(P/P_0)$.
\end{proof}
As the following example shows, if $H$ is a face of $F$ and $P_0=H\cap P$, then $P/P_0\to F/H$ is not in general a minimal free resolution.
\begin{example}
\label{ex:sliced}
\emph{
Let $P$ be the (non-simplicial) toric sharp monoid which is free on $x$, $y$, $z$, and $w$ subject to the one relation $x+y=z+w$.  Then the map $i:P\ra \N^4$ defined by 
\[
i(x)=(1,0,0,1),\quad i(y) = (0,1,1,0),\quad i(z) = (1,0,1,0),\quad i(w) = (0,1,0,1)
\]
is a minimal free resolution. Let $k$ be a field and let $f:\mathbb{A}^4_k\to \Spec k[P]$ be the induced morphism of log schemes.  Let $q:\Spec k\ra\mathbb{A}^4$ be the point corresponding to the map $k[\N^4]\ra k$ sending $e_1$ and $e_2$ to 1, and $e_3$ and $e_4$ to $0$.  The induced morphism $\bbar{\MM}_{P,f(\bar{q})}\ra \bbar{\MM}_{\N^4,\bar{q}}$ is equal to $i':P\longrightarrow\N^2$ with 
\[
i'(x)=(1,0),\quad i'(y) = (0,1),\quad i'(z) = (1,0),\quad i'(w) = (0,1).
\]
This morphism is not a minimal free resolution; it is not even injective.  However, $i'$ is a sliced resolution. This general phenomenon is the content of the following proposition which generalizes \cite[Prop 2.13]{mfr}.
}
\end{example}

\begin{proposition}
\label{prop:2.12}
Let $P$ be a  toric sharp monoid and let $i:P\to F$ be a minimal free resolution.  If $R$ is a ring, then the induced morphism $f:\Spec R[F]\to \Spec R[P]$ of log schemes is a sliced resolution.
\end{proposition}
\begin{proof}
Let $\bar{t}$ be a geometric point of $\Spec R[F]$ and let $\mf{p}$ be the corresponding prime ideal of $R[F]$.  
Let $H$ be the face of $F$ consisting of elements which map to units under 
$F\to R[F]\to R[F]_{\mf{p}}$.  Then $\bbar{\MM}_{P,f(\bar{t})}\to \bbar{\MM}_{F,\bar{t}}$ is 
given by the natural map $\eta:P/P_0\to F/H$, where $P_0=H\cap P$.  If we let $F_0$ be the face of $F$ 
generated by $P_0$, we see $F_0\subset H$ and so $\eta$ factors as 
\[
P/P_0\stackrel{\bar{\imath}}{\longrightarrow} F/F_0\longrightarrow F/H.
\]
By Proposition \ref{prop:face}, we see that $\bar{\imath}$ is a minimal free resolution.  Since $(H/F_0)\cap (P/P_0)=0$, we see 
then that $\eta$ is a sliced resolution.
\end{proof}

\begin{proposition}
\label{prop:2.17}
Let $P$ be a  toric sharp monoid and $i:P\to F$ an injective morphism to a free monoid $F$.  
Let $R$ be a ring and let $(f,h):(T,\MM)\to (\Spec R[P],\MM_P)$ 
be a morphism of fine log schemes.  If we have a commutative diagram 
\[
\xymatrix{
P\ar[r]^i\ar[d]_-{\bbar{\pi}} & F\ar[d]^\gamma\\
f^{-1}\bbar{\MM}_{P,\bar{s}}\ar[r]^{\bar{h}_{\bar{s}}} & \bbar{\MM}_{\bar{s}}
}
\]
and $\gamma$ \'etale locally lifts to a chart, then in an fppf neighborhood of $\bar{s}$, there is a chart $\epsilon:F\to\MM$ making the diagram
\[
\xymatrix{
P\ar[r]^i\ar[d]_-{\pi} & F\ar[d]^\epsilon\\
f^*\MM_P\ar[r]^h & \MM
}
\]
commute such that $\bar{\epsilon}_{\bar{s}}=\gamma$.
\end{proposition}
\begin{proof}
We follow the proof of \cite[Prop 2.17]{mfr}. We can find bases $e_1,\dots,e_r$ of $P^{gp}$ and $e'_1,\dots,e'_d$ of $F^{gp}$, and positive integers $\lambda_1,\dots,\lambda_r$ such that $i(e_j)=\lambda_je'_j$. Let $n_j\in \MM^{gp}_{\bar{s}}$ such that $\bbar{n}_j=\gamma^{gp}(e'_j)$. Then for $j\leq r$, 
\[
(h\pi)^{gp}_{\bar{s}}(e_j)=u_j+\lambda_jn_j
\]
for some $u_j\in\OO^*_{T,\bar{s}}$. After base changing to $\OO_{T,\bar{s}}[T_1,\dots,T_{r}]/(T_i^{\lambda_i}-u_i)$, we have a map $\eta:F^{gp}\to \MM^{gp}_{\bar{s}}$ defined by $\eta(e'_j)=T_j+n_j$ for $j\leq r$, and $\eta(e'_j)=n_j$ for $r<j\leq d$. Note that $\bbar{\eta}=\gamma^{gp}$ and $\eta i^{gp}=(h\pi)^{gp}_{\bar{s}}$. Let
\[
\beta:F\longrightarrow \MM^{gp}_{\bar{s}}\times_{\bbar{\MM}^{gp}_{\bar{s}}} \bbar{\MM}_{\bar{s}}=\MM_{\bar{s}}
\]
be defined by $\beta(f)=(\eta(f),\gamma(f))$. Then $\beta i = h_{\bar{s}}\pi$. Since $P$ and $F$ are finitely-generated monoids, $\beta$ extends to a chart $\epsilon$ in an fppf neighborhood of $\bar{s}$ such that $\epsilon i = h\pi$.
\end{proof}

\section{The Stacky Resolution Theorem}
\label{sec:stack}
Throughout this section, $k$ is a field and $S=\Spec k$ has the trivial log structure.  Given an fs log scheme $X$, we define a fibered category $\mf{X}$ over $X$-schemes as follows.  Objects of $\mf{X}(T)$ are pairs $(\NN,f)$, where $\NN$ is a fine log structure on $T$ and $f:(T,\NN)\rightarrow (X,\MM_X)$ is a sliced resolution whose map on underlying schemes is the structure morphism to $X$.  A morphism $(\NN,f)\ra(\NN',f')$ of objects of $\mf{X}(T)$ is a strict morphism $h:(T,\NN)\rightarrow (T,\NN')$ such that $f=f'h$.

Then $\mf{X}$ is a stack on the \'etale site of $X$ (in fact on the fppf site by \cite[Thm A.1]{log}), although \emph{a priori} it is not clear that $\mf{X}$ is an algebraic stack.  Note that upon showing $\mf{X}$ is algebraic, it inherits a natural log structure $\MM_{\mf{X}}$ defined as follows: given a morphism $g:T\ra\mf{X}$ corresponding to $(\NN,f)$, we define $g^*\MM_{\mf{X}}$ to be $\NN$.

\begin{proposition}
\label{prop:comp}
Let $P$ be a  toric sharp monoid and $i:P\to F$ a minimal free resolution.  Let $R$ be a ring and $G$ be the group scheme $D(F^{gp}/P^{gp})$ over $R$.  If $X=\Spec R[P]$, then $\mf{X}$ is isomorphic to $\mathcal{Y}:=[\Spec R[F]/G]$ over $X$.
\end{proposition}
\begin{proof}
Let $h:\mathcal{Y}\to X$ and $\pi:\Spec R[F]\to\mathcal{Y}$ be the natural morphisms.  By \cite[Prop 5.20]{log}, the stack $\mathcal{Y}$ has the following moduli interpretation.  The fiber over $f:T\to X$ is the groupoid of triples $(\NN,\eta,\gamma)$, where $\NN$ is a fine log structure on $T$, where $\gamma:F\to\bbar{\NN}$ is a morphism which \'etale locally lifts to a chart, and where $\eta:f^*\MM_P\to\NN$ is a morphism of log structures such that
\[
\xymatrix{
P\ar[r]^i\ar[d] & F\ar[d]^\gamma\\
f^{-1}\bbar{\MM}_P\ar[r]^{\bar{\eta}} & \bbar{\NN}
}
\]
commutes.

We show that $\eta$ is a sliced resolution.  Let $g:T\to\mathcal{Y}$ be the morphism representing $(\NN,\eta,\gamma)$ and let $\bar{t}$ be a geometric point of $T$.  Since $\pi$ is surjective, we have a dotted arrow making the diagram
\[
\xymatrix{
&  & \Spec R[F]\ar[d]^\pi\\
\bar{t}\ar[r]\ar@{-->}[urr] & T\ar[r]^g\ar[dr]_f & \mathcal{Y}\ar[d]^h\\
&  & \Spec R[P]
}
\]
commute.  Recall from the proof of \cite[Prop 5.20]{log} that $\eta$ is simply the pullback under $g$ of the natural morphism $h^*\MM_P\to\MM_{\mathcal{Y}}$.  Therefore, $\bar{\eta}_{\bar{t}}$ is the morphism 
\[
\bbar{\MM}_{P,h\pi(\bar{t})}=(\pi^*h^*\bbar{\MM}_P)_{\bar{t}}\longrightarrow (\pi^*\bbar{\MM}_{\mathcal{Y}})_{\bar{t}}=
\bbar{\MM}_{F,\bar{t}},
\]
which is a sliced resolution by Proposition \ref{prop:2.12}.

We therefore have a morphism $\Phi:\mathcal{Y}\to\mf{X}$ of stacks which forgets $\gamma$.  To prove full faithfulness of $\Phi$, we must show that if 
\[
(\NN_j,\eta_j:f^*\MM_P\longrightarrow\NN_j,\gamma_j:F\longrightarrow\bbar{\NN}_j)
\]
are objects of $\mathcal{Y}$ for $j=1,2$, then any isomorphism of log structures $\xi:\NN_1\to\NN_2$ such that $\xi\eta_1=\eta_2$ automatically satisfies 
$\bar{\xi}\gamma_1=\gamma_2$.  The equality $\bar{\xi}\gamma_1=\gamma_2$ can be checked on stalks.  Let $t\in T$.  Since the $\gamma_j$ \'etale locally lift to charts $\epsilon_j:F\to\NN_j$, we see that $\bbar{\NN}_{1,\bar{t}}$ and $\bbar{\NN}_{2,\bar{t}}$ are both free, as $\gamma_{j,\bar{t}}=\bar{\epsilon}_{j,\bar{t}}$ is a quotient by a face.  They have the same rank $r$ since $\bar{\xi}_{\bar{t}}$ is an isomorphism.

Let $e_1,\dots,e_d$ be the irreducible elements of $F$.  Since any two minimal free resolutions of $P$ are uniquely isomorphic over $P$, we can assume $i(p)=\sum v_i(p)e_i$, where $v_1,\dots,v_d$ are the extremal rays of $\CV$.  Let $e'_1,\dots,e'_r$ be the irreducible elements of $\bbar{\NN}_{1,\bar{t}}$ and let $e''_1,\dots,e''_r$ be the irreducible elements of $\bbar{\NN}_{2,\bar{t}}$.  Then for $j=1,2$, there exists a subset $S_j$ of $\{1,\dots,d\}$ and a bijection $\sigma_j:S_j\ra\{1,\dots,r\}$ such that
\[
\gamma_{1,\bar{t}}(e_i)=\left\{
\begin{array}{cl}
e'_{\sigma_1(i)}, & i\in S_1\\
0, & i\notin S_1
\end{array}
\right.
\quad\textrm{and}\quad
\gamma_{2,\bar{t}}(e_i)=\left\{
\begin{array}{cl}
e''_{\sigma_2(i)}, & i\in S_2\\
0, & i\notin S_2
\end{array}
\right.
\]
We also have a bijection $\tau:\{1,\dots,r\}\ra\{1,\dots,r\}$ such that $\bar{\xi}_{\bar{t}}(e'_i)=e''_{\tau(i)}$.  To show $(\bar{\xi}\gamma_1)_{\bar{t}}=\gamma_{2,\bar{t}}$, we must show $S_1=S_2$ and $\tau\sigma_1=\sigma_2$.  While we do not yet know $(\bar{\xi}\gamma_1)_{\bar{t}}=\gamma_{2,\bar{t}}$, we do know that this equality holds if we precompose with $i$.  That is, we have
\[
\sum_{i\in S_1}v_i(p)e''_{\tau\sigma_1(i)}=\sum_{i\in S_2}v_i(p)e''_{\sigma_2(i)}
\]
for all $p\in P$.  Comparing the $e''_j$ coordinates, we see
\[
v_{(\tau\sigma_1)^{-1}(j)}(p)=v_{\sigma_2^{-1}(j)}(p)
\]
for all $p\in P$.  Since the $v_i$ are the distinct extremal rays of $\CV$, we must have $(\tau\sigma_1)^{-1}(j)=\sigma_2^{-1}(j)$. Since this is true for all $j$, we have $S_1=S_2$ and $\tau\sigma_1=\sigma_2$.  This completes the proof that $\Phi$ is fully faithful.

We now prove essential surjectivity of $\Phi$.  Let $f:(T,\NN)\to(X,\MM_P)$ be a sliced resolution.  By full faithfullness, we need only show that $f$ is fppf locally in the image of $\Phi$.  Let $\bar{t}$ be a geometric point of $T$.  Then $\bbar{\MM}_{P,\bar{f(t)}}=P/P_0$ for some face $P_0$ of $P$.  If $F_0$ is the face of $F$ generated by $P_0$, then by Proposition \ref{prop:face}, the natural morphism $P/P_0\to F/F_0$ is a minimal free resolution.  Since $\bar{f}_{\bar{t}}:\bbar{\MM}_{P,\bar{f(t)}}\to \bbar{\NN}_{\bar{t}}$ is a sliced resolution by assumption, it has the form 
\[
P/P_0\longrightarrow F/F_0\longrightarrow (F/F_0)/F_1
\]
for some face $F_1$ of $F/F_0$ such that $F_1\cap (P/P_0)=0$.  Letting $H$ be the face of $F$ such that $H/F_0=F_1$, we see that
\[
P\longrightarrow f^{-1}\bbar{\MM}_{P,\bar{t}}\longrightarrow \bbar{\NN}_{\bar{t}}
\] 
factors as 
\[
P\longrightarrow F\longrightarrow F/H.
\]
Proposition \ref{prop:2.17} then shows that fppf locally, $f$ is in the image of $\Phi$.
\end{proof}
\begin{theorem}
\label{thm:main}
Let $X$ be an fs log scheme which is log smooth over $S$.  Then $\mf{X}$ is a smooth log smooth Artin stack over $S$ and the structure map $f:\mf{X}\to X$ is a good moduli space morphism.  Moreover, the base change of $f$ to the smooth locus of $X$ is an isomorphism.
\end{theorem}
\begin{proof}
To check that $\mf{X}$ is a smooth log smooth Artin stack over $S$, it suffices to look strict \'etale locally on $X$.  Since $X$ is log smooth over $S$, Theorem 4.8 and Lemma 7.1 of \cite{logdef} show that there is a strict \'etale cover $h:Y\ra X$ and a smooth strict morphism $g:Y\ra Z$, where $Z=\Spec k[P]$ with $P$ a toric sharp monoid.  Let $i:P\rightarrow F$ be a minimal free resolution.  Then Proposition \ref{prop:comp} shows that $\mf{X}\times_X Y$ is isomorphic to $\mathcal{Y}\times_Z Y$, where $G=D(F^{gp}/P^{gp})$ and $\mathcal{Y}=[\Spec k[F] / G]$.  Since $\mathcal{Y}$ is a smooth log smooth Artin stack over $S$, it follows that $\mf{X}$ is as well.

To show $f$ is a good moduli space morphism, by \cite[Prop 4.7 (ii)]{alper} it suffices to show this is true for $\mf{X}\times_X Y\ra Y$.  Since $k[F]^G=k[P]$, the map $\mathcal{Y}\ra Z$ is a good moduli space morphism by \cite[Ex 8.3]{alper}.  Then $\mathcal{Y}\times_Z Y\ra Y$ is a good moduli space morphism, as such maps are stable under base change by \cite[Prop 4.7 (i)]{alper}.

Lastly, if $x\in X^{sm}$, then $\bbar{\MM}_{X,\bar{x}}$ is free by \cite[Lemma 3.8]{mfr}. So, every sliced resolution $(T,\NN)\to (X^{sm},\MM_X|_{X^{sm}})$ is strict and therefore has no automorphisms. This shows that the base change of $f$ to $X^{sm}$ is an isomorphism.
\end{proof}
\begin{remark}
\label{rmk:tame}
\emph{We can improve slightly on \cite[Thm 3.3(2)]{mfr}.  If $X$ is a good toroidal 
embedding (\cite[1.2]{mfr}), then $\mf{X}$ is a tame Artin stack in the sense of 
\cite{tame}.  It follows from \cite[Lemma 5.5]{cst} that \cite[Thm 3.3(2)]{mfr} 
still holds for such $X$.  That is, we do not need to assume that $X$ is a tame 
toroidal embedding (\cite[1.2]{mfr}).}
\end{remark}
We end this section by showing that Iwanarai's stack of admissible free resolutions can also be generalized to the case when the charts of $X$ are given by toric monoids which are not necessarily simplicial.  Throughout the rest of this section, $X$ is an fs log scheme which is log smooth over $S$.
\begin{definition}
\label{def:admfrqfr}
\emph{Let $P$ be a toric sharp monoid, $i:P\to F$ be a minimal free resolution, and $b_i$ be positive integers indexed by the irreducible elements $F$. Then $i':P\to F'$ is a $(b_i)$-\emph{free resolution} if there is an isomorphism $\xi:F'\to F$ such that $\xi i'$ equals}
\[
P\stackrel{i}{\longrightarrow} F\stackrel{\cdot(b_i)}{\longrightarrow} F.
\]
\emph{We say that $P\to F''$ is a $(b_i)$-\emph{sliced resolution} if there is a $(b_i)$-free resolution $i':P\to F'$, a face $H$ of $F'$ such that $i'(P)\cap H=0$, and isomorphism $F''\to F'/H$ over $P$.}
\end{definition}

To define the corresponding notions for morphisms of log schemes, we first generalize \cite[Prop 3.1]{mfr}. Recall from the Notation section that $X^{triv}$ denotes the locus of $X$ where the log structure is trivial.
\begin{proposition}
\label{prop:components}
For every point $x\in X$, there is a natural map from the set $I$ of irreducible elements of the minimal free resolution of $\bbar{\MM}_{X,\bar{x}}$ to the set of irreducible components of $X\setminus X^{triv}$ that contain $x$.
\end{proposition}
\begin{proof}
It is enough to prove that $I$ is in canonical bijection with the set of irreducible components of the inverse image of $X\setminus X^{triv}$ in the strict henselization $\Spec \OO_{X,\bar{x}}$. To show this, it suffices to work with a strict \'etale neighborhood of $\bar{x}$.  By Theorem 4.8 and Lemma 7.1 of \cite{logdef}, we can therefore assume that there is a strict smooth morphism $f:X\ra \Spec k[P]$, where $P$ is a toric sharp monoid, $f(x)$ is the torus-invariant point of the toric variety $\Spec k[P]$, and the irreducible components of $X\setminus X^{triv}$ are precisely the inverse images of the torus-invariant divisors of $\Spec k[P]$.  We can therefore replace $X$ by $\Spec k[P]$ and $x$ by the torus-invariant point.  In this case, the result is clear as there is a natural bijection between the extremal rays of $\CV$ and the torus-invariant divisors of $\Spec k[P]$.
\end{proof}
\begin{definition}
\label{def:admtypebi}
\emph{Let $b_i$ be a positive integer for every irreducible component $D_i$ of $X\setminus X^{triv}$.  Given a geometric point $\bar{x}$ of $X$ with image $x\in X$, let $I(\bar{x})$ be the set of irreducible components of $X\setminus X^{triv}$ that contain $x$.  Then a morphism $f:(Y,\MM_Y)\to (X,\MM_X)$ from a fine log scheme is a $(b_i)$-\emph{free resolution}, resp.~a $(b_i)$-\emph{sliced resolution} if for all geometric points $\bar{y}$ of $Y$, the induced morphism
\[
\bbar{\MM}_{X,f(\bar{y})}\longrightarrow \bbar{\MM}_{Y,\bar{y}}
\]
is a $(b_i)_{i\in I(f(\bar{y}))}$-free resolution, resp.~a $(b_i)_{i\in I(f(\bar{y}))}$-sliced resolution.}
\end{definition}
With this definition in place, for any choice $b_i$ of positive integers indexed by the irreducible components of $X\setminus X^{triv}$, let $\mf{X}_{(b_i)}$ be the fibered category over $X$-schemes defined as follows. Objects of $\mf{X}_{(b_i)}(T)$ are pairs $(\NN,f)$, where $\NN$ is a fine log structure on $T$ and $f:(T,\NN)\rightarrow (X,\MM_X)$ is a $(b_i)$-sliced resolution whose map on underlying schemes is the structure morphism to $X$.  A morphism $(\NN,f)\ra(\NN',f')$ of objects of $\mf{X}_{(X,\MM_X)}(T)$ is a strict morphism $h:(T,\NN)\rightarrow (T',\NN')$ such that $f=f'h$.

As before, this fibered category is a stack on the fppf site by \cite[Thm A.1]{log}. After replacing the use of minimal free resolutions by $(b_i)$-free resolutions, and sliced resolution by $(b_i)$-sliced resolution, the proofs of Propositions \ref{prop:2.12} and \ref{prop:comp} yield the following two results:
\begin{proposition}
\label{prop:adm2.12}
Let $P$ be a  toric sharp monoid, $i:P\to F$ be a minimal free resolution, and $b_i$ be positive integers indexed by the irreducible elements of $F$.  If $i':P\to F'$ is a $(b_i)$-free resolution and $X=\Spec k[P]$, then the induced morphism $f:X\to\Spec k[F']$ of log schemes is a $(b_i)$-sliced resolution; here we are using the natural map of Proposition \ref{prop:components} from the set of the irreducible elements of $F$ to the set of irreducible components of $X\setminus X^{triv}$.
\end{proposition}
\begin{proposition}
\label{prop:admcomp}
Let $P$ be a  toric sharp monoid, $i:P\to F$ be a minimal free resolution, and $b_i$ be positive integers indexed by the irreducible elements of $F$.  If $i':P\to F'$ is a $(b_i)$-free resolution and $X=\Spec k[P]$, then $\mf{X}_{(b_i)}$ is isomorphic to  $[\Spec k[F']/D(F'^{gp}/i'(P^{gp}))]$ 
over $X$.
\end{proposition}
We now prove an analogue of Theorem \ref{thm:main}.  Note that if some $b_i>1$, then a $(b_i)$-sliced resolution $P\rightarrow F'$ may have automorphisms, even if $P$ is free.  As a result, the morphism $\mf{X}_{(b_i)}\ra X$ is not an isomorphism over $X^{sm}$;  however, it is an isomorphism over $X^{triv}$:
\begin{theorem}
\label{thm:admmain}
Let $b_i$ be a positive integer for every irreducible component $D_i$ of $X\setminus X^{triv}$.  Then $\mf{X}_{(b_i)}$ is a smooth log smooth Artin stack over $S$.  The natural map $\mf{X}_{(b_i)}\to X$ is a good moduli space morphism and the base change of this map to $X^{triv}$ is an isomorphism.
\end{theorem}
\begin{proof}
The proof that $\mf{X}_{(b_i)}$ is a smooth log smooth Artin stack with good moduli space $X$ is the same as in Theorem \ref{thm:main}.  Lastly, if $x\in X^{triv}$, then $\bbar{\MM}_{X,\bar{x}}=0$, so every $(b_i)$-sliced resolution of $\bbar{\MM}_{X,\bar{x}}$ is an isomorphism.  Hence, every $(b_i)$-sliced resolution $(T,\NN)\to (X^{triv},\MM_X|_{X^{triv}})$ has no automorphisms.
\end{proof}
The $\mf{X}_{(b_i)}$ are all root stacks of the stack $\mf{X}=\mf{X}_{(1)}$ in Theorem \ref{thm:main}.  As we will see in Section \ref{sec:toric}, if we restrict $X$ to being a toric variety, rather than an arbitrary log smooth log scheme, then we can construct many other moduli-theoretic smooth log smooth stacks having $X$ as a good moduli space.

\section{The Chevalley-Shephard-Todd Theorem\\for Diagonalizable Group Schemes}
\label{sec:cst}
Throughout this section, $k$ is a field and $A$ is a finitely-generated abelian group.  We let $G=D(A)$ be the 
diagonalizable group scheme over $k$ associated to $A$ and we fix a faithful action of $G$ on a finite-dimensional 
$k$-vector space $V$.  Our goal in this section is to give necessary and sufficient conditions for when the 
invariants $\kVG$ is a polynomial algebra over $k$.  In the process of working toward this goal, we give necessary and 
sufficient conditions for when $\Spec(\kVG)$ is a simplicial toric variety, see Theorem \ref{thm:msopsimp} (also obtained by Wehlau in \cite[Thm 4.1]{wehlau}).

\begin{definition}
\label{def:stable}
\emph{
An action of a torus $T$ on a finite-dimensional $k$-vector space $W$ is called \emph{stable} if there is a dense open subset of $\Spec k[W]$ which is a disjoint union of closed $T$-orbits.  An action of a diagonalizable group scheme $G$ on $W$ is called \emph{stable} if the induced action of its maximal torus is stable.
}
\end{definition}
Choosing $V'\subset V$ as in \cite[Lemma 2]{wehlau2}, we see that $V'$ is invariant under $G$ and stable.  Furthermore, we have $\kVG=k[V']^G$.  Replacing $G$ by the quotient of the kernel of the representation on $V'$, we can assume that the action on $V'$ is faithful and stable.  So, for the purposes of studying $k[V]^G$, we may assume the $G$-action on $V$ is faithful and stable. We assume this is the case throughout the rest of this section.

\begin{definition}
\label{def:MSOP}
\emph{Let $M$ be a finitely-generated abelian group and $T=D(M)$ the associated torus over $k$. A faithful $T$-action on a finite-dimensional $k$-vector space $W$ with weights $m_1,\dots,m_d\in M$ is called \emph{orderly} if there are weights $m_{i_1},\dots,m_{i_r}$ which are a basis for $M\otimes_{\Z}\Q$ and such that all other weights are non-positive rational linear combinations of the $m_{i_j}$.  We say that $m_{i_1},\dots,m_{i_r}$ is an \emph{orderly basis}.  An action of a diagonalizable group scheme $G$ on $W$ is called \emph{orderly} if the induced action of its maximal torus is orderly.}
\end{definition}

We can now state the first main result of this section.

\begin{theorem}
\label{thm:msopsimp}
The $G$-action on $V$ is orderly if and only if $\Spec(\kVG)$ is a simplicial toric variety .
\end{theorem}

To prove Theorem \ref{thm:msopsimp}, we need the following two technical lemmas.

\begin{lemma}
\label{l:stable}
There is a morphism $\pi:F'\ra A$ from a free monoid such that $k[V]=k[F']$ and the $G$-action on $k[V]$ agrees with action induced by $\pi$.  Furthermore, there is an injective exact morphism $i':P\ra F'$ from a toric sharp monoid $P$ such that $A=F'^{gp}/i'(P^{gp})$ and the induced morphism $k[P]\ra k[F']^G$ is an isomorphism.
\end{lemma}
\begin{proof}
Since $G$ is diagonalizable, we can write $V=\bigoplus V_i$, where the $V_i$ are one-dimensional subrepresentations.  Let $e'_i$ be a basis of $V_i$, and let $F'$ be the free monoid on the $e'_i$.  The $G$-action on $V$ yields an $A$-grading, and hence a map $\pi:F'\ra A$.  It is clear that $k[V]=k[F']$ and that the $G$-action on $k[V]$ agrees with the one induced by $\pi$.  The map $\pi^{gp}:F'^{gp}\ra A$ is surjective, as the $G$-action on $V$ is faithful.

Since the $G$-action on $k[F']$ is linear, it is a toric action.  Let $P\subset F'$ be the submonoid of monomials such that $k[P]=k[F']^G$, and let $i':P\ra F'$ denote the inclusion.  Note that the torus $T':=\Spec k[F'^{gp}]$ of $\Spec k[F']$ is a disjoint union of closed $T$-orbits, where $T=D(A/A_{tors})$ is the maximal torus of $G$.  Indeed, if $T'$ contains a $T$-orbit $Z$ which is not closed, then translating $Z$ by $T'$ covers $T'$ by $T$-orbits which are not closed, violating the assumption that the $G$-action on $V$ is stable.  Similarly, we see that the $G$-action on $T'$ is free.  It follows that the torus $\Spec k[P^{gp}]$ of $\Spec k[P]$ is equal to $\Spec k[F'^{gp}]^G$, and so $A=F'^{gp}/i'(P^{gp})$.  Lastly, to see $i'$ is exact, note that if $f'\in F'$ has the same image in $F'^{gp}$ as an element of $P^{gp}$, then $\pi(f')=0$, and so $f'\in P$.
\end{proof}

Throughout the rest of this section, we fix the following notation.  Let $i:P\ra F$ be a minimal free resolution and let $e_1,\dots,e_d$ be the irreducible elements of $F$.  Then
\[
i(p)=\sum_i v_i(p)e_i,
\]
where $v_1,\dots,v_d$ are the first lattice points on the extremal rays of $\CV$.  Let $e'_1,\dots,e'_{d'}$ be the irreducible elements of $F'$ and let $w_1,\dots,w_{d'}\in\Hom(P,\N)$ such that 
\[
i'(p)=\sum_jw_j(p)e'_j
\]
for all $p\in P$.  As in the proof of Proposition \ref{prop:can}, exactness of $i'$ shows that
\[
\Cone(w_1,\dots,w_{d'})=\CV=\Cone(v_1,\dots,v_d).
\]
Let $M=A/A_{tors}$, so that $D(M)$ is the maximal torus of $G$, and let $m_i$ be the image of $\pi(e'_i)$ in $M$.

\begin{lemma}
\label{l:msopbasis}
Let $S\subset\{1,\dots,d'\}$.  Then $\{m_i\}_{i\in S}$ is an orderly basis if and only if there exists a basis $\{p_j\}_{j\notin S}$ of $P^{gp}\otimes\Q$ with each $p_j\in P$ satisfying $w_j(p_j)>0$, and $w_\ell(p_j)=0$ if $j,\ell\notin S$ and $j\neq \ell$.
\end{lemma}
\begin{proof}
We first assume there exists a basis $\{p_j\}_{j\notin S}$ of $P^{gp}\otimes\Q$ as above and show $\{m_i\}_{i\in S}$ is an orderly basis.  Let $a_j=w_j(p_j)>0$ and let $a_{ij}=w_i(p_j)\geq0$ for $i\in S$.  Since the $p_j$ form a basis for $P^{gp}\otimes\Q$, we see that
\[
|S|=d'-\rank P^{gp},
\]
which is the rank of $M$, by Lemma \ref{l:stable}.  To prove $\{m_i\}_{i\in S}$ is a basis of $M\otimes\Q$, we therefore need only show that these elements are linearly independent.  If $\sum_{i\in S}c_im_i=0$ for some $c_i\in\Q$, then since 
\[
0\longrightarrow P^{gp}\otimes\Q\stackrel{{i'}^{gp}}{\longrightarrow} F'^{gp}\otimes\Q\stackrel{\pi^{gp}}{\longrightarrow} A\otimes\Q\longrightarrow 0
\]
is an exact sequence, we have
\[
\sum_{i\in S}c_ie'_i=i'(p)
\]
in $F'^{gp}\otimes\Q$ for some $p\in P^{gp}\otimes\Q$.  Writing $p=\sum_{j\notin S}c'_jp_j$ for $c'_j\in\Q$, and comparing the $e'_j$ coordinates above, we see $c'_ja_j=0$, which shows that $p=0$, and so the $c_i=0$ as well.  This shows that $\{m_i\}_{i\in S}$ is a basis of $M\otimes\Q$.

To show that the $\{m_i\}_{i\in S}$ form an orderly basis, note that for $j\notin S$,
\[
a_je'_j+\sum_{i\in S}a_{ij}e'_i=\sum_{\ell=1}^{d'}w_\ell(p_j)e'_\ell=i'(p_j).
\]
Since the $a_{ij}\geq0$, we see then that $m_j$ is a non-positive rational linear combination of the $m_i$ for $i\in S$.

We now prove the ``only if'' direction of the lemma.  If $j\notin S$, then $m_j$ is a non-positive rational linear combination of the $m_i$ for $i\in S$.  We therefore have some positive integer $a_j$, non-negative integers $a_{ij}$, and some $p_j\in P^{gp}$ such that
\[
a_je'_j=-\sum_{i\in S}a_{ij}e'_i+\sum_{\ell=1}^{d'}w_\ell(p_j)e'_\ell.
\]
We see then that $w_j(p_j)=a_j$, $w_i(p_j)=a_{ij}$ if $i\in S$, and $w_\ell(p_j)=0$ otherwise.  Since $w_{\ell}(p_j)$ is non-negative for all $\ell$, exactness of 
$i'$ shows that $p_j\in P$.

We prove that the $p_j$ form a basis for $P^{gp}\otimes\Q$.  Since the rank of $P^{gp}$ is $d'-|S|$, we need only show that the $p_j$ are linearly 
independent.  If $\sum_{j\notin S}c_jp_j=0$ for $c_j\in \Q$, then applying $w_{j_0}$ for $j_0\notin S$ shows that $c_{j_0}a_{j_0}=0$.  Since $a_{j_0}$ is 
positive, we have $c_{j_0}=0$, as desired.
\end{proof}

\begin{proof}[Proof of Theorem \ref{thm:msopsimp}]
Since $\Cone(w_1,\dots,w_{d'})=\CV=\Cone(v_1,\dots,v_d)$, after possibly reordering the $w_j$, we can assume  that $w_j$ is a positive multiple of $v_j$ for $j\leq d$.

Suppose first that $P$ is simplicial.  Then $i:P\to F$ is close, so for all $j\leq d$, there exists $p_j\in P$ and a positive integer $\lambda_j$ such that $i(p_j)=\lambda_je_j$.  Since $i(p)=\sum v_i(p)e_i$ and $P$ is simplicial, we see that the $p_j$ form a basis for $P^{gp}\otimes\Q$, that $v_j(p_j)=\lambda_j>0$, and that $v_\ell(p_j)=0$ for $j,\ell\leq d$ with $j\neq \ell$.  Since $w_j$ is a positive multiple of $v_j$ for $j\leq d$, we see that Lemma \ref{l:msopbasis} finishes the proof of the ``if'' direction.

Assume now that the $G$-action on $V$ is orderly.  Let $S$ and $\{p_j\}_{j\notin S}$ be as in Lemma \ref{l:msopbasis}, and let $a_j=w_j(p_j)$.  To prove the ``only if'' direction, it suffices to prove the stronger assertion that the extremal rays of $\CV$ are precisely the rational rays defined by the $w_j$ for $j\notin S$, with each $w_j$ lying on a distinct extremal ray. Indeed, $P$ is then simplicial, as $d=d'-|S|=\rank P^{gp}$.

We begin by showing that the $w_j$ for $j\notin S$ lie on extremal rays.  Fix $j_0\notin S$ and suppose that $w_{j_0}$ does not lie on an extremal ray.  Then we have
\[
w_{j_0}=\sum_{\ell=1}^r c'_\ell v_{i_\ell}
\]
for positive rational numbers $c'_\ell$, distinct $i_\ell$, and $r\geq2$.  If $j\notin S$ and $j\neq j_0$, then $w_{j_0}(p_j)=0$.  Since $p_j\in P$, we see the $v_{i_\ell}(p_j)\geq0$, and so $v_{i_\ell}(p_j)=0$.  Since the $p_j$ for $j\notin S$ form a basis for $P^{gp}\otimes\Q$, we see that the $v_{i_\ell}$ are determined by their values on $p_{j_0}$.  They therefore all define the same extremal ray, which is a contradiction.

To prove that the $w_j$ for $j\notin S$ lie on distinct extremal rays, assume there exist distinct $j$ and $\ell$ with $j,\ell\notin S$ and a positive rational number $c$ such that $w_\ell=cw_j$.  Then we see
\[
0=w_\ell(p_j)=cw_j(p_j)=ca_j>0
\]
which is a contradiction.

Lastly, we show that every extremal ray is in fact rationally generated by some $w_j$ for $j\notin S$.  Since $\Cone(w_1,\dots,w_{d'})=\CV$, every extremal ray of $\CV$ must be rationally generated by some $w_i$.  Since the $w_j$ for $j\notin S$ lie on distinct extremal rays, they define a basis for the $\Q$-vector space $\Hom(P^{gp},\Z)\otimes\Q$.  Hence for $i\in S$, we have $c'_{ij}\in\Q$ such that 
\[
w_i=\sum_{j\notin S}c'_{ij}w_j.
\]
Since $0\leq w_i(p_j)=c'_{ij}a_j$, we see that the $c'_{ij}\geq0$.  Therefore, if $w_i$ generates an extremal ray $\rho$, then $\rho$ is rationally generated by some $w_j$ for $j\notin S$.
\end{proof}

\begin{proposition}
\label{prop:msop}
If $\kVG$ is a polynomial algebra, then the $G$-action on $V$ is orderly.
\end{proposition}
\begin{proof}
If $\kVG$ is polynomial, then $P$ is free.  Since $\Cone(w_1,\dots,w_{d'})=
\Cone(v_1,\dots,v_d)$, after possibly reordering the $w_j$, we can assume that for $j\leq d$ there is a positive integer $c_j$ such that $w_j=c_jv_j$, and that for $j>d$ there are $c_{ij}\in\Q_{\geq0}$ such that
\[
w_j=\sum_{i=1}^d c_{ij}v_i.
\]
Then $i'^{gp}:P^{gp}\ra F'^{gp}$ is given by the $d'\times d$ matrix whose top $d\times d$ block is diagonal with $j$-th diagonal entry given by $c_j$, and whose bottom $(d'-d)\times d$ block is given by $(c_{ij})_{i,j}$.  Therefore, we see that $m_{d+1},\dots,m_{d'}$ is a basis for $M\otimes\Q$ and that for $j\leq d$, 
\[
m_j=-\sum_{i>d} \frac{c_{ij}}{c_j}m_i
\]
is a non-positive rational linear combination of $m_{d+1},\dots,m_{d'}$.
\end{proof}

We now turn to the second main result of this section, the Chevalley-Shephard-Todd theorem for diagonalizable group schemes.

\begin{theorem}
\label{thm:cst}
If the $G$-action on $V$ is orderly, then fix an orderly basis $m_1,\dots,m_\ell$ and let $F'_0$ be the face generated by the elements of $F'$ which map to positive linear combinations of $m_1,\dots,m_\ell$.  If $B$ denotes the image of ${F_0'}^{gp}$ in $A$ under $\pi^{gp}$, then $A/B$ is a finite abelian group and the induced action of $D(A/B)$ on $k[F'/F'_0]$ has the property that
\[
\kVG=k[F'/F'_0]^{D(A/B)};
\]
therefore $\kVG$ is polynomial if and only if the action of $D(A/B)$ on $k[F'/F'_0]$ is generated by pseudo-reflections, as defined in \emph{\cite[Def 1.2]{cst}}.
\end{theorem}
\begin{proof}
Without loss of generality, we can assume $m_{d+1},\dots,m_{d'}$ is an orderly basis.  Then the proof of Theorem \ref{thm:msopsimp} shows that $\Cone(w_1,\dots,w_d)=\CV$.  After reordering $w_1,\dots,w_d$ if necessary, we can assume that for $j\leq d$ there are positive integers $c_j$ such that $w_j=c_jv_j$, and for $j>d$ there are $c_{ij}\in\Q_{\geq0}$ such that
\[
w_j=\sum_{i=1}^d c_{ij}v_i.
\]
Note that $F'_0$ is the face of $F'$ generated by the $e'_i$ for $i>d$.  Letting $n$ be a positive integer such that $nc_{ij}$ is an integer for all $i,j$, we have a commutative diagram
\[
\xymatrix{
F\ar[r]^\psi & F'\ar[r]^-{\pi'} & F'/F'_0\\
P\ar[u]^i\ar[r]^{i'} & F'\ar[u]^{\cdot n}\ar[r]^-{\pi'} & F'/F'_0\ar[u]^{\cdot n}
}
\]
Here $\pi'$ is the natural projection and 
\[
\psi(e_i)=nc_ie'_i+\sum_{j>d}nc_{ij}e'_j.
\]
We see then that $\pi'\psi(e_i)=nc_ie'_i$, and so $\pi'\psi$ is exact.  Since the multiplication by $n$ map from $F'/F'_0$ to itself is exact, \cite[Prop I.4.1.3(2)]{ogus} shows that $\pi'i'$ is exact as well.  Hence,
\[
\kVG=k[P]=k[F'/F'_0]^{D((F'/F'_0)^{gp}/\pi'i'(P^{gp}))}.
\]
Furthermore, since $P$ is simplicial by 
Theorem \ref{thm:msopsimp}, and since $F'/F'_0$ is free of the same rank as $P^{gp}$, we see that $(F'/F'_0)^{gp}/\pi'i'(P^{gp})$ is a finite group.  Note that by the definition of $B$, there exists a morphism $\varphi$ making the diagram 
\[
\xymatrix{
F'^{gp}\ar[r]^-{\pi'}\ar[d] & F'^{gp}/{F'_0}^{gp}\ar[d]^{\varphi}\\
A\ar[r] & A/B
}
\]
commute.  One easily checks that the kernel of $\varphi$ is $\pi'i'(P^{gp})$, and so $A/B$ is isomorphic to $(F'/F'_0)^{gp}/\pi'i'(P^{gp})$.  This proves that $A/B$ is a finite group and that
\[
\kVG=k[F'/F'_0]^{D(A/B)}
\]
as desired.  Lastly, since $D(A/B)$ is a finite diagonalizable group scheme, we see from Theorem 1.6 and Proposition 2.2 of \cite{cst} that 
$\kVG$ is polynomial if and only if the action of $D(A/B)$ on $k[F'/F'_0]$ is generated by pseudo-reflections.  
\end{proof}

\begin{example}
\label{ex:cst2}
\emph{
Let $V$ be a 3-dimensional vector space with basis $a,b,c$.  Consider the action of $G=\mathbb{G}_m\times\mu_2$ on $V$ described as follows.  Let $\mathbb{G}_m$ acts on $a,b,c$ with weights $1,1,-2$, respectively, let $\mu_2$ act trivially on $b$ and $c$, and act on $a$ via the non-trivial character.  One checks that $k[V]^G=k[a^2c,b^2c]$, which is polynomial.}

\emph{Note that $-2$ is an orderly basis for the $G$-action.  Let $F'$ be the free monoid on $a,b,c$.  Then $A=\Z\oplus\Z/2$ and $\pi:F'\to A$ is given by 
\[
\pi(a)=(1,1),\quad \pi(b)=(1,0),\quad \pi(c)=(-2,0).
\]
We see $F'_0$ is the submonoid of $F'$ generated by $c$, and $B=2\Z\oplus 0\subset A$.  The induced morphism
\[
F'/F'_0\longrightarrow A/B\simeq\Z/2\oplus\Z/2
\]
sending $\bar{a}$ to $(1,1)$ and $\bar{b}$ to $(1,0)$ yields an action of $\mu_2\times\mu_2$ on $k[F'/F'_0]$.  When the characteristic of $k$ is not 2, we can identify $\mu_2\times\mu_2$ with $\Z/2\times\Z/2$, and we see that $(0,1)$ and $(1,1)$ are pseudo-reflections that generate $\Z/2\times\Z/2$.  More generally, when $k$ has arbitrary characteristic, the action of $\mu_2\times\mu_2$ is generated by the pseudo-reflections $1\times\mu_2$ and the diagonal copy of $\mu_2$ (see \emph{\cite[Def 1.2]{cst}} for the definition of pseudo-reflections for actions of finite linearly reductive group schemes).  Hence, Theorem \ref{thm:cst} predicts $\kVG$ is polynomial, which is the case as we saw above.  We remark that
\[
k[V]^{\mathbb{G}_m}=k[a^2c,abc,b^2c]\simeq k[x,y,z]/(xz-y^2),
\]
which is not polynomial.
}
\end{example}

\section{Toric Artin Stacks}
\label{sec:toric}
Throughout this section, let $k$ be a field endowed with the trivial log structure. In Subsection \ref{subsec:stackyfan}, we generalize the definition of stacky fan given in \cite{bcs} and associate to a generalized stacky fan $\mathbf{\Sigma}=(N,\Sigma,\beta)$ a smooth log smooth Artin stack $\mf{X}(\mathbf{\Sigma})$.  
In Subsection \ref{subsec:Xgms}, we prove that $X(\Sigma)$ is the good moduli space of $\mf{X}(\mathbf{\Sigma})$. In Subsection \ref{subsec:qmfrtype}, we show that under suitable conditions, $\mf{X}(\mathbf{\Sigma})$ has a log geometric moduli interpretation.

\subsection{Generalized Stacky Fans and their associated Toric Artin Stacks}
\label{subsec:stackyfan}
Given an abelian group $A$, we let $A^*=\Hom(A,\Z)$.  Given a finitely-generated abelian group $N$ and a rational fan $\Sigma$ on $N\otimes\Q$, we let $\Sigma(1)$ be the set of rays of the fan.  We denote by $d$ the rank of $N$ and $n$ the order of $\Sigma(1)$.  Let $M=N^*$.

We introduce the following notion of a generalized stacky fan.  We frequently drop the word ``generalized'' when referring to it.
\begin{definition}
\label{def:genstackyfan}
\emph{A \emph{generalized stacky fan} $\mathbf{\Sigma}$ consists of a finitely-generated abelian group $N$, a rational fan $\Sigma$ on $N\otimes\Q$ such that the rays of $\Sigma$ span $N\otimes \Q$, 
a choice of $r\in\N$, and a morphism $\beta:\Z^{\Sigma(1)}\times\Z^r\to N$.  
Let $e_\rho$ be the generator of $\Z^{\Sigma(1)}$ corresponding to a ray $\rho\in\Sigma(1)$.  
We require that $\beta(e_\rho)\otimes1$ lie on the ray $\rho$ and that $\beta$ maps the standard generators of $\Z^r$ to the support of $\Sigma$.  We often suppress $r$ and write $\mathbf{\Sigma}=(N,\Sigma,\beta)$.
}
\end{definition}
Throughout this subsection and the next, we fix for every stacky fan $\mathbf{\Sigma}=(N,\Sigma,\beta)$ an ordering on the rays of $\Sigma(1)$ so that $\beta$ is a map from $\Z^{n+r}$.  In Subsection \ref{subsec:qmfrtype}, however, canonicity will be more important.
\begin{remark}
\label{rmk:artinstackytorus}
\emph{
Note that in Definition \ref{def:genstackyfan} we do not require 
that the $\beta(e_j)$ be distinct.  Some of the $\beta(e_j)$ can even be zero, which 
corresponds to the associated stack $\mf{X}(\mathbf{\Sigma})$ containing a copy of $[\mathbb{A}^1/\mathbb{G}_m]$.
}
\end{remark}
\begin{remark}
\label{rmk:jiang}
\emph{
In \cite[\S2]{jiang}, Jiang introduces a notion of extended stacky fans which is equivalent to our definition above, but 
he requires that $\Sigma$ be simplicial and 
the stacks he associates to extended stacky fans are all Deligne-Mumford.  His goal is to obtain suitable presentations of toric Deligne-Mumford stacks rather than construct toric Artin stacks.
}
\end{remark}
We show now how to associate to a stacky fan $\mathbf{\Sigma}=(N,\Sigma,\beta)$ a log Artin stack $(\mf{X}(\mathbf{\Sigma}),\MM_{\mf{X}(\mathbf{\Sigma})})$.  We follow the procedure in \cite{bcs}.  Consider the exact triangle
\[
\Z^{n+r}\stackrel{\beta}{\longrightarrow} N\longrightarrow \Cone(\beta)
\]
in the derived category.  Applying $R\Hom(-,\Z)$ and taking cohomology, we obtain an exact sequence
\[
N^*\stackrel{\beta^*}{\longrightarrow} (\Z^{n+r})^*\longrightarrow H^1(\Cone(\beta)^*)\longrightarrow\Ext^1_{\Z}(N,\Z)\longrightarrow0.
\]
We define $\beta^{\vee}:(\Z^{n+r})^*\to DG(\beta)$ to be the connecting homomorphism above and let $DG(\beta):=H^1(\Cone(\beta)^*)$.

\begin{remark}
\label{rmk:concrete-constr}
\emph{As in the last paragraph of \cite[p.195]{bcs}, we have the following concrete description of $DG(\beta)$ and $\beta^{\vee}$. Let}
\[
0\longrightarrow \Z^{\ell}\stackrel{Q}{\longrightarrow} \Z^{d+\ell}\longrightarrow N\longrightarrow0
\]
\emph{be a free resolution of $N$.  If $B:\Z^{n+r}\to \Z^{d+\ell}$ is a lift of $\beta$, then}
\[
DG(\beta)=\coker([BQ]^*)
\]
\emph{and $\beta^{\vee}$ is the composite}
$(\Z^{n+r})^*\to (\Z^{n+r+\ell})^*\to DG(\beta)$.
\end{remark}

We construct $\mf{X}(\mathbf{\Sigma})$ as the quotient of an open subscheme of $\mathbb{A}^{n+r}$.  Consider the ideal
\[
J_{\mathbf{\Sigma}}=\langle \prod_{\beta(e_i)\otimes1\notin\sigma}\!\!\!\!\!\!x_i \mid \sigma\in\Sigma\rangle
\]
of $k[x_1,\dots,x_{n+r}]$.  If we let $G_{\mathbf{\Sigma}}$ be the diagonalizable group scheme associated to $DG(\beta)$, then via $\beta^{\vee}$ we have a morphism $G_{\mathbf{\Sigma}}\to\mathbb{G}_m^{n+r}$.  Restricting the usual action of $\mathbb{G}_m^{n+r}$, we obtain an action of $G_{\mathbf{\Sigma}}$ on $\mathbb{A}^{n+r}=\Spec k[x_1,\dots,x_{n+r}]$.
Since $V(J_{\mathbf{\Sigma}})$ is a union of coordinate subspaces, we see that $Z_{\mathbf{\Sigma}}:=\mathbb{A}^{n+r}\setminus V(J_{\mathbf{\Sigma}})$ is $G_{\mathbf{\Sigma}}$-invariant.  We obtain a log structure $\MM_{Z_{\mathbf{\Sigma}}}$ on $Z_{\mathbf{\Sigma}}$ by pulling back the standard log structure on $\mathbb{A}^{n+r}$.  Note that the $G_{\mathbf{\Sigma}}$-action on $Z_{\mathbf{\Sigma}}$ extends to the log scheme $(Z_{\mathbf{\Sigma}},\MM_{Z_{\mathbf{\Sigma}}})$.  We define 
\[
\mf{X}(\mathbf{\Sigma})=[Z_{\mathbf{\Sigma}}/G_{\mathbf{\Sigma}}]
\]
and obtain a log structure $\MM_{\mf{X}(\mathbf{\Sigma})}$ on $\mf{X}(\mathbf{\Sigma})$ by descent.  We see then that $\mf{X}(\mathbf{\Sigma})$ is smooth and that $(\mf{X}(\mathbf{\Sigma}),\MM_{\mf{X}(\mathbf{\Sigma})})$ is log smooth.
\begin{definition}
\label{def:associatedtoric}
\emph{
The log stacks $(\mf{X}(\mathbf{\Sigma}),\MM_{\mf{X}(\mathbf{\Sigma})})$ obtained from stacky fans $\mathbf{\Sigma}$ via the above construction are called \emph{toric Artin stacks}.
}
\end{definition}
We end this subsection with a series of examples.
\begin{example}
\emph{If $\mathbf{\Sigma}=(0,0,\beta:\Z^d\to0)$, then $\mf{X}(\mathbf{\Sigma})=[\mathbb{A}^d/\mathbb{G}_m^d]$, which is a smooth toric Artin stack in the sense of Lafforgue \cite[IV.1.a]{lafforgue}.}
\end{example}
\begin{example}
\label{ex:simpgms}
\emph{
Let $N=\Z^2$ and let $\sigma$ be the cone generated by the rays $(1,0)$ and $(1,2)$.  Let $\Sigma$ be the rational fan consisting of $\sigma$ and its faces, so that the toric variety $X(\Sigma)$ is the spectrum of $R:=k[x,y,z]/(xy-z^2)$.  Consider the morphism
\[
\beta:\Z^3\longrightarrow N
\]
\[
\beta(e_1)=(1,0),\quad \beta(e_2)=(1,2),\quad \beta(e_3)=(1,1).
\]
If $\mathbf{\Sigma}=(N,\Sigma,\beta)$, then $\mf{X}(\mathbf{\Sigma})=[\mathbb{A}^3/\mathbb{G}_m]$, where $\mathbb{G}_m$ acts with weights $1,1,-2$.  As we saw in Example \ref{ex:cst2}, the $\mathbb{G}_m$-invariants of this action is given by $\Spec R$; that is, $X(\Sigma)$ is the good moduli space of $\mf{X}(\mathbf{\Sigma})$.  We will see in Theorem \ref{thm:toricgood} that this a general phenomenon.
}
\end{example}
In Theorem \ref{thm:main} we proved that every toroidal embedding $X$ has a stacky resolution.  If $X$ is a toric variety, then we can describe this stacky resolution as a toric stack:
\begin{example}
\label{ex:cantoricstack}
\emph{
Given a normal toric variety $X$ defined by a rational fan $\Sigma$ on $N$, let $\beta:\Z^{\Sigma(1)}\ra N$ send $e_\rho$ to the first lattice point on $\rho$.  Let $\mathbf{\Sigma}^{can}:=(N,\Sigma,\beta)$.  We will see in Theorem \ref{thm:twostacks} that $\mf{X}(\Sigma^{can})$ and the stack of Theorem \ref{thm:main} are isomorphic as log stacks over $X$.}

\emph{If $X(\Sigma)$ is a simplicial toric variety, $\mf{X}$ is a toric Deligne-Mumford stack, and $\pi:\mf{X}\ra X(\Sigma)$ is a coarse space map, then $\pi$ necessarily factors through $\mf{X}(\Sigma^{can})\ra X(\Sigma)$.  This is one part of the classification result in \cite{fant} $($\emph{cf.} ~\cite{iwtoric}$)$ known as the ``bottom-up construction.''}

\emph{We remark that such a ``bottom-up'' result does \emph{not} hold for toric Artin stacks.  For example, if $X(\Sigma)$ and $\mf{X}(\mathbf{\Sigma})$ are as in Example \ref{ex:simpgms}, then the good moduli space morphism $\pi:\mf{X}(\mathbf{\Sigma})\ra X(\Sigma)$ does not factor through $\mf{X}(\Sigma^{can})\ra X(\Sigma)$.  One way to see this is to apply Theorem \ref{thm:twostacks} and use the fact that $\pi$ is not a sliced resolution.
}
\end{example}
Let us now look at example where $\Sigma$ is a non-simplicial fan.  
\begin{example}
\label{ex:atiyahflop}
\emph{
Consider the cone $\sigma$ on $\Z^3$ generated by $(0,0,1),(1,0,1),(0,1,1),(1,1,1)$ and let $\Sigma$ be the fan consisting of the faces of $\sigma$.  There are two toric small resolutions of $X=X(\Sigma)$: one is given by triangulating $\Sigma$ by adding the two-dimensional face generated by $(0,0,1),(1,1,1)$; the other is obtained by triangulating $\Sigma$ by adding the two-dimensional face generated by $(1,0,1),(0,1,1)$.  Let $\pi_i:X_i\ra X$ denote these two resolutions.  We have a commutative diagram}
\[
\xymatrix{
X_1\ar[r]^{j_1}\ar[dr]_{\pi_1} & \mf{X}\ar[d]^\pi & X_2\ar[l]_{j_2}\ar[dl]^{\pi_2}\\
 & X &
}
\]
\emph{where the $j_i$ are open immersions and $\pi:\mf{X}\ra X$ is the stacky resolution of Theorem \ref{thm:main}.  To see this, by Example \ref{ex:cantoricstack} we know $\mf{X}=\mf{X}(\Sigma^{can})=[\mathbb{A}^4/\mathbb{G}_m]$ where $\mathbb{G}_m$ acts with weights $1,1,-1,-1$.  Letting $a,b,c,d$ be the coordinates of $\mathbb{A}^4$, we see that the two small resolutions are given by $[(\mathbb{A}^4\setminus Z_i)/\mathbb{G}_m]$, where $Z_1=V(c,d)$ and $Z_2=V(a,b)$.}
\end{example}

\subsection{$X(\Sigma)$ is the Good Moduli Space of $\mf{X}(\mathbf{\Sigma})$}
\label{subsec:Xgms}
The main result of this subsection is Theorem \ref{thm:toricgood} which shows that if $\mathbf{\Sigma}=(N,\Sigma,\beta)$ is a stacky fan, then $X(\Sigma)$ is the good moduli space of $\mf{X}(\mathbf{\Sigma})$.  As in Subsection \ref{subsec:stackyfan}, we fix for every stacky fan $\mathbf{\Sigma}=(N,\Sigma,\beta)$ an ordering on the rays of $\Sigma(1)$ so that $\beta$ is a map from $\Z^{n+r}$.  For each standard generator $e_i$ of $\Z^{n+r}$, we let $w_i$ denote the image of $\beta(e_i)$ in $N\otimes\Q$.  We denote the function $M:=N^*\ra\Z$ given by pairing with $w_i$ by $w_i(-)$.

Let $F_{\mathbf{\Sigma}}$ be the free monoid on the standard basis vectors $e_i^*$ of $(\Z^{n+r})^*$.  Then, identifying $\mathbb{A}^{n+r}$ with $\Spec k[F_{\mathbf{\Sigma}}]$, the action of $G_{\mathbf{\Sigma}}$ on $\mathbb{A}^{n+r}$ obtained via 
\[
F_{\mathbf{\Sigma}}\longrightarrow (\Z^{n+r})^*\stackrel{\beta^\vee}{\longrightarrow} DG(\beta)
\]
agrees with the action defined in the construction of $\mf{X}(\mathbf{\Sigma})$; here $F_{\mathbf{\Sigma}}\ra (\Z^{n+r})^*$ is the natural inclusion.  For each $\sigma\in\Sigma$, let
\[
x^\sigma=\prod_{w_i\notin\sigma}\!x_i
\]
and let $U_\sigma=\mathbb{A}^{n+r}\setminus V(x^{\sigma})$.  We see that $U_\sigma$ is $G_{\mathbf{\Sigma}}$-invariant and that $Z_{\mathbf{\Sigma}}$ is the union of the $U_\sigma$.  If we let $F_\sigma$ be the submonoid of $(\Z^{n+r})^*$ defined by
\[
F_\sigma=\{\sum_{i=1}^{n+r} a_ie_i^*\;\mid\; a_i\in\N\textrm{\ if\ }w_i\in\sigma\},
\]
then we see $U_\sigma=\Spec k[F_\sigma]$ and that the inclusion of monoids $F_{\mathbf{\Sigma}}\subset F_\sigma$ induces the natural open immersion $U_\sigma\ra\mathbb{A}^{n+r}$.

Lastly, let $P_\sigma=\sigma^\vee\cap M$, so that $X_\sigma=\Spec k[P_\sigma]$ is the corresponding torus-invariant open affine of $X(\Sigma)$.  Consider the map
\[
i_\sigma:P_\sigma\longrightarrow F_\sigma
\]
\[
p\longmapsto \sum_{i=1}^{n+r} w_i(p)e_i^*
\]
This induces a morphism $(U_\sigma,\MM_{U_\sigma})\ra(X_\sigma,\MM_{X_\sigma})$, which we abusively also denote by $i_\sigma$.  Here, $\MM_{U_\sigma}$ resp.~$\MM_{X_\sigma}$ is the restriction of the log structure $\MM_{\mathbb{A}^{n+r}}$ resp.~$\MM_{X(\Sigma)}$.  Note that the diagram
\[
\xymatrix{
P_\sigma\ar[r]^{i_\sigma}\ar[d] & F_\sigma\ar[d] & \\
N^*\ar[r]^-{\beta^*} & (\Z^{n+r})^*\ar[r]^-{\beta^\vee} & DG(\beta)
}
\]
commutes, where the vertical maps are the natural inclusions, and the bottom row is exact.  Therefore, the morphism $k[P_\sigma]\ra k[F_\sigma]$ induced by $i_\sigma$ factors through $k[F_\sigma]^{G_{\mathbf{\Sigma}}}$, and so we obtain a morphism
\[
\pi_\sigma:([U_\sigma/G_{\mathbf{\Sigma}}],\MM_{[U_\sigma/G_{\mathbf{\Sigma}}]})\longrightarrow (X_\sigma,\MM_{X_\sigma}),
\]
where $\MM_{[U_\sigma/G_{\mathbf{\Sigma}}]}$ is the log structure descended from $\MM_{U_\sigma}$.
\begin{lemma}
\label{l:isigmaexact}
The morphism $i_\sigma:P_\sigma\ra F_\sigma$ is exact.
\end{lemma}
\begin{proof}
Let $\xi\in P_{\sigma}^{gp}$ and suppose $i_{\sigma}^{gp}(\xi)\in F_\sigma$.  Then $w_i(\xi)\geq0$ for all $w_i\in\sigma$.  Since $P_\sigma=\sigma^\vee\cap M$ and every ray of $\sigma$ contains at least one $w_i$, we see $\xi\in P_\sigma$.
\end{proof}
\begin{proposition}
\label{prop:toriclocalgood}
The map $\pi_\sigma:[U_\sigma/G_{\mathbf{\Sigma}}]\to X_\sigma$ is a good moduli space morphism.
\end{proposition}
\begin{proof}
Since $i_\sigma$ is exact by Lemma \ref{l:isigmaexact}, we have
\[
k[P_\sigma]=k[F_\sigma]^{D(F_{\sigma}^{gp}/i_{\sigma}^{gp}(P_{\sigma}^{gp}))}
\]
We first prove the proposition in the case when $N$ is torsion-free.  Choose $\ell=0$ and $B=\beta$ as in Remark \ref{rmk:concrete-constr} so that the diagram
\[
\xymatrix{
P_\sigma^{gp}\ar[r]^{i_{\sigma}^{gp}}\ar[d]_{id} & F_\sigma^{gp}\ar[d]^{id}\\
N^*\ar[r]^-{B^*} & (\Z^{n+r})^*
}
\]
commutes.  This shows that the cokernel of $B^*$, namely $DG(\beta)$, is equal to $F_{\sigma}^{gp}/i_{\sigma}^{gp}(P_{\sigma}^{gp})$.  Therefore, 
$k[P_\sigma]=k[F_\sigma]^{G_{\mathbf{\Sigma}}}$
and so $\pi_\sigma$ is a good moduli space morphism.

We now handle the case when $N$ is not torsion-free.  Let $N_{tors}\simeq\bigoplus\Z/m_i\Z$, where the $m_i>1$.  Let $N'=N/N_{tors}$ and $\mathbf{\Sigma}'=(N',\Sigma,\beta')$ where $\beta'$ is the composite of $\beta$ and the projection of $N$ to $N'$.  There exists a free resolution of the form
\[
0\longrightarrow \Z^\ell \stackrel{Q}{\longrightarrow} N'\oplus\Z^\ell\longrightarrow N\longrightarrow 0
\]
with $Q=(0,\textrm{diag}(m_i))$.  Letting $B:N'\oplus\Z^\ell\to N$ be a lift of $\beta$, Remark \ref{rmk:concrete-constr} shows that 
\[
(N')^*\stackrel{(B')^*}{\longrightarrow} (\Z^{n+r})^*\stackrel{(\beta')^{\vee}}{\longrightarrow} DG(\beta')\longrightarrow 0
\]
is an exact sequence. Let $\pi:N'\oplus\Z^\ell\to N'$ be the projection and note that $\beta'=\pi B$. Applying Remark \ref{rmk:concrete-constr} to $\beta'$ and the free resolution $N'\stackrel{id}{\to} N'$, we obtain a commutative diagram
\[
\xymatrix{
(N')^*\ar[r]^{(B')^*}\ar[d]_{\pi^*} & (\Z^{n+r})^*\ar[r]^{(\beta')^{\vee}}\ar[d]\ar[dr]^{\beta^{\vee}} & DG(\beta')\ar@{-->}[d]^{\eta}\ar[r] & 0\\
(N'\times\Z^{\ell})^*\ar[r]^{(BQ)^*} & (\Z^{n+r+\ell})^*\ar[r] & DG(\beta)\ar[r] & 0
}
\]
with exact rows.  We therefore have a map $\eta$ making the diagram commute.  The left and middle vertical arrows are injective.  One easily checks that the left square is cartesian, and so $\eta$ is injective.  Applying the Cartier dualization functor $D(-)$, we have a commutative diagram
\[
\xymatrix{
G_{\mathbf{\Sigma}'}\ar[r]^-{D({\beta'}^\vee)} & \mathbb{G}_m^{n+r}\\
G_{\mathbf{\Sigma}}\ar[u]^-{D(\eta)}\ar[ur]_-{D({\beta}^\vee)} & 
}
\]
with $D(\eta)$ surjective.  Since the action of $G_{\mathbf{\Sigma}}$ resp.~$G_{\mathbf{\Sigma}'}$ on $U_\sigma$ is obtained by restricting the $\mathbb{G}_m^{n+r}$-action via $D(\beta^\vee)$ resp.~$D({\beta'}^\vee)$, we see that the induced map
\[
k[F_\sigma]^{G_{\mathbf{\Sigma}'}}\longrightarrow k[F_\sigma]^{G_{\mathbf{\Sigma}}}
\]
is an isomorphism.  Since the map $k[P_\sigma]\ra k[F_\sigma]^{G_{\mathbf{\Sigma}'}}$ induced by $i_\sigma$ is an isomorphism (by the case when $N$ is torsion-free), we see that $\pi_\sigma$ is a good moduli space morphism.
\end{proof}
We now prove the main theorem of this subsection:
\begin{theorem}
\label{thm:toricgood}
If $\mathbf{\Sigma}=(N,\Sigma,\beta)$ is a stacky fan, then $X(\Sigma)$ is the good moduli space of $\mf{X}(\mathbf{\Sigma})$.
\end{theorem}
\begin{proof}
We prove the stronger assertion that there is a unique morphism $\pi_{\mathbf{\Sigma}}:(\mf{X}(\mathbf{\Sigma}),\MM_{\mf{X}(\mathbf{\Sigma})})\to (X(\Sigma),\MM_{X(\Sigma)})$ such that
\[
\xymatrix{
([U_\sigma/G_{\mathbf{\Sigma}}],\MM_{[U_\sigma/G_{\mathbf{\Sigma}}]})\ar[r]\ar[d]_{\pi_\sigma} & (\mf{X}(\mathbf{\Sigma}),\MM_{\mf{X}(\mathbf{\Sigma})})\ar[d]^{\pi_{\mathbf{\Sigma}}}\\
(X_\sigma,\MM_{X_\sigma})\ar[r] & (X(\Sigma),\MM_{X(\Sigma)})
}
\]
is cartesian for all $\sigma\in\Sigma$; here, the horizontal maps are the natural open immersions. Since $\pi_\sigma$ is a good moduli space morphism by Proposition \ref{prop:toriclocalgood}, it follows from Lemma 6.3 and Proposition 7.9 of \cite{alper} that $\pi_{\mathbf{\Sigma}}$ is as well.

Let $\sigma\in\Sigma$ and let $\tau$ be a face of $\sigma$.  Then the diagram
\[
\xymatrix{
(U_\tau,\MM_{U_\tau})\ar[r]\ar[d]_{i_\tau} & (U_\sigma,\MM_{U_\sigma})\ar[d]^{i_\sigma}\\
(X_\tau,\MM_{X_\tau})\ar[r] & (X_\sigma,\MM_{X_\sigma})
}
\]
commutes, where the horizontal maps are the natural open immersions.  We claim that the diagram is cartesian.  To prove this, we show that if we have a commutative diagram of monoids
\[
\xymatrix{
Q & & \\
 & F_\tau\ar@{-->}[ul] & F_\sigma\ar[l]\ar[ull]_{\phi}\\
 & P_\tau\ar[u]_{i_\tau}\ar[uul]^{\psi} & P_\sigma\ar[u]^{i_\sigma}\ar[l]
}
\]
then there is a unique dotted arrow making the diagram commute.  This is equivalent to showing that if $w_i$ is in $\sigma$ but not in $\tau$, then $\phi(e_i)$ is a unit.  By \cite[\S1.2 Prop 2]{fulton}, there is some $p\in P_\sigma$ such that $\tau=\sigma\cap p^{\perp}$ and $P_\tau=P_\sigma + \N\cdot(-p)$; here, this is an equality of submonoids of $M$.  Note then that $\psi(p)$ is a unit and that 
\[
\psi(p)=\phi\,i_\sigma(p)=\sum_j w_j(p)\phi(e_j).
\]
Let $i$ be such that $w_i$ is in $\sigma$ but not in $\tau$.  Since $w_i\in\sigma$, we see $w_i(p)\geq0$.  Since $w_i$ is not in $\tau$ and since $\tau=\sigma\cap p^{\perp}$, we must have $w_i(p)>0$, and so $\phi(e_i)$ is a unit, as desired.

We see then that
\[
\xymatrix{
([U_\tau/G_{\mathbf{\Sigma}}],\MM_{[U_\tau/G_{\mathbf{\Sigma}}]})\ar[r]\ar[d]_{\pi_\tau} & ([U_\sigma/G_{\mathbf{\Sigma}}],\MM_{[U_\sigma/G_{\mathbf{\Sigma}}]})\ar[d]^{\pi_\sigma}\\
(X_\tau,\MM_{X_\tau})\ar[r] & (X_\sigma,\MM_{X_\sigma})
}
\]
is cartesian.  It follows that there is a unique morphism $\pi_{\mathbf{\Sigma}}:(\mf{X}(\mathbf{\Sigma}),\MM_{\mf{X}(\mathbf{\Sigma})})\ra (X(\Sigma),\MM_{X(\Sigma)})$ whose base change to $(X_\sigma,\MM_{X_\sigma})$ is $\pi_\sigma$.
\end{proof}

\subsection{A Moduli Interpretation of $\mf{X}(\mathbf{\Sigma})$}
\label{subsec:qmfrtype}
We begin this subsection by associating to a stacky fan $\mathbf{\Sigma}=(N,\Sigma,\beta)$ with $N$ torsion-free and distinct $\beta(e_i)$, a smooth log smooth Artin stack $\mf{X}_{\mathbf{\Sigma}}$ having $X(\Sigma)$ as a good moduli space.  The stack $\mf{X}_{\mathbf{\Sigma}}$ is constructed as a moduli space along the same lines as in Theorems \ref{thm:main} and \ref{thm:admmain}.  We then show in Theorem \ref{thm:twostacks} that $\mf{X}(\mathbf{\Sigma})$ is isomorphic to $\mf{X}_{\mathbf{\Sigma}}$ as log stacks over $X(\Sigma)$, thereby giving a moduli interpretation to $\mf{X}(\mathbf{\Sigma})$.

Although the definition of $\mf{X}_{\mathbf{\Sigma}}$ makes sense for any stacky fan, Theorem \ref{thm:twostacks} is false without the above assumptions on $\mathbf{\Sigma}$. The stacks $\mf{X}_{\mathbf{\Sigma}}$ that we construct have trivial generic stabilizer; to ensure the same is true for $\mf{X}(\mathbf{\Sigma})$, we must assume $N$ is torsion-free. The assumption that the $\beta(e_i)$ are distinct is necessary in showing that $\mf{X}(\mathbf{\Sigma})$ and $\mf{X}_{\mathbf{\Sigma}}$ are locally isomorphic (see Proposition \ref{prop:toriccomp}, which is the analogue of Proposition \ref{prop:comp}).


Note that if $\mathbf{\Sigma}=(N,\Sigma,\beta)$ is a stacky fan with $N$ torsion-free, then giving the map $\beta$ is equivalent to choosing a positive integer $b_\rho$ for every $\rho\in\Sigma(1)$ and choosing for every $j\in\{1,2,\dots,r\}$ an element $w_j\in N$ which lies in the support of $\Sigma$.  Given this equivalence, throughout this subsection, we denote stacky fans by $\mathbf{\Sigma}=(N,\Sigma;b_\rho;w_j)$.

As in the previous subsection, we denote by $w_j(-)$ the function $M:=N^*\ra\Z$ given by pairing with $w_j$.  Given $\rho\in\Sigma(1)$, let $v_\rho$ denote the first lattice point on $\rho$, and let $w_\rho=b_\rho v_\rho$.

We now work toward defining the morphisms which $\mf{X}_{\mathbf{\Sigma}}$ parameterizes.
\begin{definition}
\emph{If $P$ is a toric sharp monoid, then a \emph{datum} $D$ \emph{for} $P$ is a choice of $r\in\N$, a positive integer $b_\rho$ for every ray $\rho$ of $\CV$, and a morphism $w_j:P\to\N$ for every $j\in\{1,2,\dots,r\}$.  We frequently suppress $r$ and write $D=(b_\rho;w_j)$.}
\end{definition}
Given a stacky fan $\mathbf{\Sigma}=(N,\Sigma;b_\rho;w_j)$ with $N$ torsion-free, then for every $x\in X:=X(\Sigma)$, we obtain an associated datum of $\bbar{\MM}_{X,\bar{x}}$ as follows.  Let $I(\bar{x})$ be the set of irreducible components of the inverse image of $X\setminus X^{triv}$ in $\Spec\OO_{X,\bar{x}}$ and let 
\[
H_{\bar{x}}=\alpha_{\bar{x}}^{-1}(\OO_{X,\bar{x}}^*),
\]
where $\alpha:\MM_X\to\OO_X$ is the structure morphism of the log structure $\MM_X$.  If $w_j(H_{\bar{x}})=0$, then we have an induced morphism $\bar{w}_j:\bbar{\MM}_{X,\bar{x}}\to\N$.  We can therefore define a datum for $\bbar{\MM}_{X,\bar{x}}$ by
\[
D_{\mathbf{\Sigma},\bar{x}}=(b_\rho\textrm{\ s.t.\ }\rho\in I(\bar{x});\bar{w}_j\textrm{\ s.t.\ }w_j(H_{\bar{x}})=0).
\]
\begin{remark}
\label{rmk:simplerformulation}
\emph{
The definition of $D_{\mathbf{\Sigma},\bar{x}}$ can be stated equivalently as follows.  The point $x$ lies in some torus orbit, which corresponds to the interior of a cone $\sigma\in \Sigma$.  Then $w_j(H_{\bar{x}})=0$ if and only if the lattice point $w_j$ lies in $\sigma$, in which case the function $w_j(-)$ assumes non-negative values on $\bbar{\MM}_{X,\bar{x}}=\sigma^\vee\cap M$.  The function $\bbar{\MM}_{X,\bar{x}}\ra\N$ obtained by restricting $w_j(-)$ is $\bar{w}_j$.  Therefore, $D_{\mathbf{\Sigma},\bar{x}}$ simply consists of those $w_j\in\sigma$ and those $b_\rho$ for which $v_\rho\in\sigma$.
}
\end{remark}
\begin{definition}
\label{def:admtoric}
\emph{If $P$ is a toric sharp monoid and $D=(b_\rho;w_j)$ is a datum for $P$, then we say a morphism $i':P\to F'$ is a $D$-\emph{free resolution} if there is a minimal free resolution $i:P\to F$ and an isomorphism $\xi:F'\to F\oplus\N^r$ making the diagram}
\[
\xymatrix{
F\ar[r]^{\cdot(b_\rho)} & F\\
P\ar[u]^i\ar[r]^-{\xi i'}\ar[dr]_{(w_j)} & F\oplus\N^r\ar[u]\ar[d]\\
 & \N^r
}
\]
\emph{commute; the two arrows out of $F\oplus\N^r$ are the natural projections.  Note we are implicitly using that the rays of $\CV$ are in canonical bijection with the irreducible elements of $F$.  We say that $P\to F''$ is a $D$-\emph{sliced resolution} if there is a $D$-free resolution $i':P\to F'$, a face $H$ of $F\oplus\N^r$ such that $i'(P)\cap H=0$, and an isomorphism $F''\to F'/H$ over $P$.}
\end{definition}

\begin{definition}
\label{def:admstackyfan}
\emph{If $\mathbf{\Sigma}$ is a stacky fan with $N$ torsion-free and if $X=X(\Sigma)$, then a morphism $f:(Y,\MM_Y)\to (X,\MM_X)$ from an fs log scheme is a $\mathbf{\Sigma}$-\emph{sliced resolution} if for all geometric points $y$ of $Y$, the induced morphism $\bbar{\MM}_{X,f(y)}\to \bbar{\MM}_{Y,y}$ is a $D_{\mathbf{\Sigma},f(y)}$-sliced resolution. Here we are using the proof of Proposition \ref{prop:components} to identify $I(f(y))$ with the set of irreducible elements of the minimal free resolution of $\bbar{\MM}_{X,f(y)}$.}
\end{definition}

Note that 
if $\mathbf{\Sigma}$ is a stacky fan for which $r=0$ (\emph{i.e}.~a stacky fan in the sense of \cite{bcs}) and for which $N$ is torsion-free, then a morphism 
of log schemes is a $\mathbf{\Sigma}$-sliced resolution if and only if it is a $(b_\rho)$-sliced resolution in the sense of Definition \ref{def:admtypebi}.

Given a stacky fan $\mathbf{\Sigma}$, let $X=X(\Sigma)$.  We define a fibered category $\mf{X}_{\mathbf{\Sigma}}$ over $X$-schemes as follows.  
Objects of $\mf{X}_{\mathbf{\Sigma}}(T)$ are pairs $(\NN,f)$, where $\NN$ is a fine log structure on $T$ and $f:(T,\NN)\to (X,\MM_X)$ is a $\mathbf{\Sigma}$-sliced resolution whose map on underlying schemes is the structure morphism to $X$.  A morphism $(\NN,f)\ra(\NN',f')$ of objects of $\mf{X}_{\mathbf{\Sigma}}(T)$ is a strict morphism $h:(T,\NN)\to (T',\NN')$ such that $f=f'h$.

It follows from \cite[Thm A.1]{log} that $\mf{X}_{\mathbf{\Sigma}}$ is a stack on the fppf site of $X$.  The proof that $\mf{X}_{\mathbf{\Sigma}}$ is algebraic is similar to the proofs of Theorems \ref{thm:main} and \ref{thm:admmain}, so 
we indicate only where changes are necessary.
\begin{proposition}
\label{prop:toric2.12}
Let $\mathbf{\Sigma}=(N,\Sigma;b_\rho;w_j)$ be a stacky fan such that $N$ is torsion-free and $X:=X(\Sigma)=\Spec k[P]$, where $P$ is a toric sharp monoid.  Let $D=D_{\mathbf{\Sigma},\bar{x}}$, where $x\in X$ is the torus-invariant point.  If $i':P\to F'$ is a $D$-free resolution, then the induced morphism $f:X\to\Spec k[F']$ of log schemes is a $\mathbf{\Sigma}$-sliced resolution.
\end{proposition}
\begin{proof}
Let $i:P\ra F$ be a minimal free resolution. It suffices to prove the proposition after replacing $i'$ by an isomorphic map, so we can assume $F'=F\oplus\N^r$ and $i'(p)=(i(p), w_1(p),\dots,w_r(p))$.  Let $\bar{t}$ be a geometric point of $\Spec k[F']$ whose image in $\Spec k[F']$ is the prime ideal $\mf{p}$ of $k[F]$.  Let $H''=H\oplus H'$ be the face of $F'=F\oplus\N^r$ consisting of elements which map to units under $F'\to k[F']\to k[F']_{\mf{p}}$.  Then $\bbar{\MM}_{P,f(\bar{t})}\to \bbar{\MM}_{F',\bar{t}}$ is given by the map $\bar{\imath}'$ making the diagram
\[
\xymatrix{
P\ar[r]^{i'}\ar[d] & F'\ar[d]\\
P/P_0\ar[r]^-{\bar{\imath}'} & F/H\oplus\N^r/H'
}
\]
commute, where $P_0=i'(P)\cap H''$.  We must show that $\bar{\imath}'$ is a $D_{\mathbf{\Sigma},\bar{t}}$-sliced resolution.  Note first that if $H'_0$ denotes the face of $\N^r$ generated by the $e_j$ with $w_j(P_0)\neq0$, then commutativity of the above diagram shows that $H'_0\subset H'$.  Similarly, we see that $F_0\subset H$, where $F_0$ denotes the face of $F$ generated by $i(P_0)$.  As a result, we have a commutative diagram
\[
\xymatrix{
P\ar[r]^{i'}\ar[d] & F'\ar[dr]\ar[d]^{\pi}\\
P/P_0\ar[r]^-{i''} & F/F_0\oplus\N^r/H'_0\ar[r]^{\pi'} & F/H\oplus\N^r/H'
}
\]
where the bottom row composes to $\bar{\imath}'$, and $\pi$ and $\pi'$ are the natural projections.  By Proposition \ref{prop:face}, the natural morphism from $P/P_0$ to $F/F_0$ is a minimal free resolution.  Note that $i''(\bar{p})=(w_\rho(p);w_j(p))$ for $\rho$ such that $v_\rho(P_0)=0$ and $j$ such that $w_j(P_0)=0$; that is, $i''$ is a $D_{\mathbf{\Sigma},\bar{t}}$-free resolution.

To complete the proof, we must show $i''(P/P_0)\cap(H/F_0\oplus H'/H'_0)=0$.  This amounts to showing that if $w_\rho(p)e_\rho\in H$ for all $\rho$ such that $v_\rho(P_0)=0$ and if $w_j(p)e_j\in H'$ for all $w_j(P_0)=0$, then $w_\rho(p)e_\rho\in H$ for all $\rho$ and $w_j(p)e_j\in H'$ for all $j$.  If $\rho$ is such that $v_\rho(P_0)\neq0$, then $e_\rho\in F_0\subset H$ and so $w_\rho(p)e_\rho\in H$.  Similarly, if $j$ is such that $w_j(P_0)\neq0$, then $e_j\in H'_0\subset H'$ and so $w_j(p)e_j\in H'$.  This shows the above intersection is trivial.
\end{proof}

The proofs of Proposition \ref{prop:comp} and Theorem \ref{thm:main} then yield the following two results. Note the assumption in Proposition \ref{prop:toriccomp} that $b_\rho v_\rho,w_1,\dots,w_r$ be distinct is necessary in showing fullness of the map $[\Spec k[F']/G]\to\mf{X}_{\mathbf{\Sigma}}$.
\begin{proposition}
\label{prop:toriccomp}
Let $\mathbf{\Sigma}=(N,\Sigma;b_\rho;w_j)$ be a stacky fan such that $\Sigma$ consists of the faces of a single cone $\sigma$, $N$ is torsion-free, and the $b_\rho v_\rho,w_1,\dots,w_r$ are distinct.  Let $X:=X(\Sigma)=\Spec k[P]$, where $P=\sigma^\vee\cap M$.  Then $\mf{X}_{\mathbf{\Sigma}}$ is an Artin stack.  Moreover, if $x\in X$ is the torus-fixed point, $i':P\to F'$ is a $D_{\mathbf{\Sigma},\bar{x}}$-free resolution, and $G=D(F'^{gp}/P^{gp})$, then $\mf{X}_{\mathbf{\Sigma}}$ is isomorphic to $[\Spec k[F']/G]$ over $X$.
\end{proposition}

\begin{theorem}
\label{thm:toricmain}
If $\mathbf{\Sigma}=(N,\Sigma;b_\rho;w_j)$ is a stacky fan with $N$ torsion free and $b_\rho v_\rho,w_1,\dots,w_r$ distinct, then $\mf{X}_{\mathbf{\Sigma}}$ is a smooth log smooth Artin stack over $k$ with good moduli space $X(\Sigma)$.  
\end{theorem}

We end by showing that $(\mf{X}(\mathbf{\Sigma}),\MM_{\mf{X}(\mathbf{\Sigma})})$ is isomorphic to the moduli space of $\mathbf{\Sigma}$-sliced resolutions.
\begin{theorem}
\label{thm:twostacks}
If $\mathbf{\Sigma}=(N,\Sigma;b_\rho;w_j)$ is a stacky fan with $N$ torsion-free and $b_\rho v_\rho,w_1,\dots,w_r$ distinct, then $\mf{X}(\mathbf{\Sigma})$ and $\mf{X}_{\mathbf{\Sigma}}$ are isomorphic as log stacks over $X(\Sigma)$.
\end{theorem}
\begin{proof}
Let $X=X(\Sigma)$.  We begin by showing that $\pi_{\mathbf{\Sigma}}:(\mf{X}(\mathbf{\Sigma}),\MM_{\mf{X}(\mathbf{\Sigma})})\ra(X,\MM_X)$ is a $\mathbf{\Sigma}$-sliced resolution.  It suffices to check this Zariski locally on $X$.  Let $\sigma\in\Sigma$ be a maximal cone and let $X_\sigma$ be the corresponding torus-invariant affine open of $X$.  Let $U_\sigma$, $F_\sigma$, $P_\sigma$, $i_\sigma$, and $\pi_\sigma$ be as in Subsection \ref{subsec:Xgms}.  Let
\[
S=\Sigma(1)\cup\{1,\dots,r\}\quad\textrm{and}\quad S_\sigma=\{\rho\in\Sigma(1)\textrm{\ s.t.\ }v_\rho\in\sigma\}\cup\{j\textrm{\ s.t.\ }w_j\in\sigma\}.
\]
Then
\[
F_\sigma=\{\sum_{s\in S}c_se_s\mid c_s\in\N\textrm{\ if\ }s\in S_\sigma\}\subset\Z^S.
\]
Note that the units $Q$ of the monoid $F_\sigma$ are $\Z^{S\setminus S_\sigma}$, so the composite $\bar{\imath}_\sigma:P_\sigma\stackrel{i_\sigma}{\to} F_\sigma\stackrel{\epsilon}{\to}\bbar{F}_\sigma$ is given by
\[
\bar{\imath}_\sigma(p)=\sum_{s\in S_\sigma}w_s(p)e_s.
\]
Therefore, $\bar{\imath}_\sigma$ is a $D$-free resolution, where $x\in X_\sigma$ is the torus-invariant point and $D=D_{\mathbf{\Sigma},\bar{x}}$.  By Proposition \ref{prop:toric2.12}, we see then that
\[
(\Spec k[\bbar{F}_\sigma],\MM_{\bbar{F}_\sigma})\stackrel{\epsilon}{\longrightarrow} (U_\sigma,\MM_{F_\sigma})\stackrel{i_\sigma}{\longrightarrow} (X_\sigma,\MM_{P_\sigma})
\]
is a $D$-sliced resolution.  Since $\epsilon$ is strict, $i_\sigma$ is a $D$-sliced resolution.  Since the base change of $\pi_{\mathbf{\Sigma}}$ to $X_\sigma$ is $\pi_\sigma:([U_\sigma / G_{\mathbf{\Sigma}}], \MM_{[U_\sigma / G_{\mathbf{\Sigma}}]})\ra (X_\sigma,\MM_{P_\sigma})$, it follows that $\pi_{\mathbf{\Sigma}}$ is a $\mathbf{\Sigma}$-sliced resolution.

We therefore obtain a morphism
\[
(\mf{X}(\mathbf{\Sigma}),\MM_{\mf{X}(\mathbf{\Sigma})})\longrightarrow (\mf{X}_{\mathbf{\Sigma}},\MM_{\mf{X}_{\mathbf{\Sigma}}})
\]
over $(X,\MM_X)$, which we claim is an isomorphism.  We can check this Zariski locally, so we may base change to $X_\sigma$.  By Proposition \ref{prop:toriccomp}, we have
\[
\mf{X}_{\mathbf{\Sigma}}\times_X X_\sigma\simeq [\Spec k[\bbar{F}_\sigma] / D(\bbar{A})]
\]
as log stacks, where $\bbar{A}=\coker(\bar{\imath}^{gp}_\sigma)$.  Note that $G_{\mathbf{\Sigma}}=D(A)$, where $A=\coker(i^{gp}_\sigma)$.  We must therefore show that the log map
\[
[\Spec k[\bbar{F}_\sigma] / D(\bbar{A})]\longrightarrow [\Spec k[F_\sigma] / D(A)]
\]
induced by $\epsilon:F_\sigma\ra\bbar{F}_\sigma$ is an isomorphism.  Since
\[
[\Spec k[\bbar{F}_\sigma] / D(\bbar{A})]=[(\Spec k[\bbar{F}_\sigma]\times^{D(\bbar{A})}D(A)) / D(A)],
\]
it suffices to show the log map $\Spec k[\bbar{F}_\sigma]\times^{D(\bbar{A})}D(A)\ra \Spec k[F_\sigma]$ induced by $\epsilon$ is an isomorphism.

Applying the Snake Lemma to
\[
\xymatrix{
0 \ar[r] & Q
\ar[r] & F_\sigma^{gp}\ar[r] & \bbar{F}_\sigma^{gp}\ar[r] & 0\\
0 \ar[r] & 0\ar[r]\ar[u] & P_\sigma^{gp}\ar[r]^{id}\ar[u]^{i_\sigma^{gp}} & P_\sigma^{gp}\ar[r]\ar[u]_{\bar{\imath}_\sigma^{gp}} & 0
}
\]
we obtain a short exact sequence
\[
\xymatrix{
0 \ar[r] & Q
\ar[r] & A\ar[r]^\beta & \bbar{A}\ar[r] & 0.
}
\]
Let $\gamma:F_\sigma\ra A$ and $\delta:\bbar{F}_\sigma\ra\bbar{A}$ be the natural morphisms.  By definition,
\[
\Spec k[\bbar{F}_\sigma]\times^{D(\bbar{A})}D(A)=\Spec k[\bbar{F}_\sigma\oplus A]^{D(A')}
\]
where the action of $D(A')$ is induced by the morphism
\[
\eta:\bbar{F}_\sigma\oplus A\longrightarrow A'
\]
\[
(\bar{f},a)\longmapsto \delta(\bar{f})-\beta(a).
\]
Letting $Q'\subset \bbar{F}_\sigma\oplus A$ be the submonoid of elements $(\bar{f},a)$ which are mapped to 0 under $\eta$, we have
\[
\Spec k[\bbar{F}_\sigma]\times^{D(\bbar{A})}D(A)=\Spec k[Q']
\]
as log schemes.  It therefore suffices to show that the map
\[
\phi:F_\sigma\longrightarrow Q
\]
\[
f\longmapsto (\epsilon(f),\gamma(f))
\]
is an isomorphism.

We first show $\phi$ is injective.  Suppose $\phi(f)=\phi(f')$.  Then $\epsilon(f)=\epsilon(f')$ and so $f'=f+q$ for some $q\in Q$.  Since $q+\beta(f)=\beta(f')=\beta(f)$, we see $q=0$ as desired.

Lastly, we show $\phi$ is surjective.  Let $(\bar{f},a)\in Q'$ and let $f\in F_\sigma$ such that $\epsilon(f)=\bar{f}$.  Then there exists $q\in Q$ such that $a=\delta(\bar{f})+q$.  We see then that $\phi(f+q)=(\bar{f},a)$.
\end{proof}

\section*{Acknowledgements}  
I would like to thank Dan Abramovich, Bhargav Bhatt, Ishai Dan-Cohen, Anton Geraschenko, Arthur Ogus, and Shenghao Sun for many helpful conversations.  It is a pleasure to thank my advisor, Martin Olsson, whose guidance and inspiration greatly shaped this paper. Lastly, I thank the anonymous referee for his enthusiastic comments and helpful suggestions.

\end{document}